\documentclass{amsart}
\usepackage{tikz-cd}
\usepackage{amsmath}
\usepackage{mathtools}
\usepackage{amsmath,amssymb,amsthm,mathrsfs}
\usepackage[utf8]{inputenc}
\usepackage[T1]{fontenc}
\usepackage{hyperref}
\usepackage{enumitem}

\numberwithin{equation}{section}
\theoremstyle{plain} 
\newtheorem{thm}{Theorem}
\numberwithin{thm}{section}
\newtheorem*{thm*}{Theorem}
\newtheorem{prop}[thm]{Proposition}
\newtheorem{lemma}[thm]{Lemma}
\newtheorem{coro}[thm]{Corollary}

\theoremstyle{definition}
\newtheorem{definition}[thm]{Definition}

\newtheorem{remark}[thm]{Remark}
\newtheorem*{ack}{Acknowledgment}

\newcommand{\C}{\mathbb{C}}

\renewcommand{\H}{\mathbb{H}}

\newcommand{\Q}{\mathbb{Q}}
\newcommand{\R}{\mathbb{R}}

\newcommand{\Z}{\mathbb{Z}}

\newcommand{\cA}{\mathcal{A}}

\newcommand{\cE}{\mathcal{E}}

\newcommand{\cH}{\mathcal{H}}

\newcommand{\cJ}{\mathcal{J}}
\newcommand{\cK}{\mathcal{K}}
\newcommand{\cL}{\mathcal{L}}
\newcommand{\cM}{\mathcal{M}}
\newcommand{\cN}{\mathcal{N}}
\newcommand{\cO}{\mathcal{O}}
\newcommand{\cP}{\mathcal{P}}
\newcommand{\cQ}{\mathcal{Q}}

\newcommand{\cT}{\mathcal{T}}

\newcommand{\cV}{\mathcal{V}}

\newcommand{\cX}{\mathcal{X}}
\newcommand{\cY}{\mathcal{Y}}

\newcommand{\bV}{\mathbf{V}}

\newcommand{\ba}{\mathbf{a}}
\newcommand{\bb}{\mathbf{b}}
\newcommand{\bc}{\mathbf{c}}
\newcommand{\bd}{\mathbf{d}}
\newcommand{\be}{\mathbf{e}}
\newcommand{\bk}{\mathbf{k}}

\newcommand{\0}{\mathbf{0}}
\newcommand{\1}{\mathbf{1}}

\newcommand{\fe}{\mathfrak{e}}

\newcommand{\fp}{\mathfrak{p}}

\newcommand{\sA}{\mathscr{A}}

\newcommand{\sD}{\mathscr{D}}

\newcommand{\sR}{\mathscr{R}}

\newcommand{\sX}{\mathscr{X}}

\newcommand{\DR}{\textup{DR}}
\newcommand{\Gr}{\textup{Gr}}
\newcommand{\Sp}{\textup{Sp}}
\newcommand{\Dol}{\textup{Dol}}
\newcommand{\Pic}{\textup{Pic}}

\newcommand{\num}{\textup{num}}

\DeclareMathOperator{\red}{red}

\DeclareMathOperator{\vv}{\vee\vee}

\title{Kodaira-type Vanishings via nonabelian Hodge Theory}
\author{Chuanhao Wei}

\begin{document}

\begin{abstract}
    In this paper, we use nonabelian Hodge theory to study Kodaira-type vanishings. In particular, we generalize Saito vanishing using mixed twistor $\sD$-modules. We also generalize it to a Kawamata-Viehweg type vanishing using $\Q$-divisors, and prove an effective global generation theorem, as an application. We also prove a relative vanishing for a projective morphism in the mixed twistor $\sD$-module setting.
\end{abstract}

\maketitle
\section{Introduction}
Let $X$ be a complex projective manifold. Given a reduced normal crossing divisor 
$D=D_1+...+D_n$ on $X$, and a vector $\bd=(d_1,...,d_n)\in \R^n,$ we set $\bd D=d_1D_1+...+d_nD_n$. We say $\bd\geq \ba$, if $d_i\geq a_i$ for all $i$, similarly defined for other inequalities. We use $\1=(1,...,1)\in \R^n$, similarly for $\0$. We also use $aD$ to denote $aD_1+...+aD_n$, for $a\in \R$. A divisor denoted by a single letter, e.g. $D$, is an effective integral divisor, unless there is a coefficient in front of it, like $\bd D, aD,$ as above, or explicitly stated otherwise. We can just consider the $\Q$-divisors, and the $\R$-divisor case can be reduced to such case by an argument appeared later in Remark \ref{R: R div}.
 
The renowned Kodaira-Akizuki-Nakano Vanishing Theorem has the following generalized form, which is due to Esnault-Viehweg \cite{EV86}, \cite{EV92}, combined with the Kawamata-Viehweg type Vanishing formulation \cite{AMPW}:
\begin{thm}\label{T: Nakano Vanishing}
Notations as above, assume that we have an ample divisor $A$. Let $L$ be an integral divisor, such that
$$L \equiv_{\num} aA+\bd D,$$
with $a>0$, $\0\leq\bd\leq \1$.
Then, we have the following vanishings:
\begin{align*}
    H^i(X, \Omega^j_X(\log D)(-D)\otimes \cO_X(L))=0&, \text{for } i+j>\dim(X);\\
    H^i(X, \Omega^j_X(\log D)\otimes \cO_X(-L))=0&, \text{for } i+j<\dim(X).
\end{align*}
\end{thm} 

There is another direction to generalize Kodaira-Akizuki-Nakano Vanishing using Saito's mixed Hodge Module. In this paper, we say a pair $(\cM, F_\bullet)$ \emph{underlies a mixed Hodge Module}, where $\cM$ is a \emph{right} coherent $\sD$-module, and $F_\bullet$ is its increasing Hodge filtration, if we can add a weight filtration to make it an algebraic graded polarizable mixed Hodge module in the sense of \cite[\S 2]{Sa90}. 
We remark that, among other things, the existence of the admissible weight and the polarization on each weight-graded piece on $(\cM, F_\bullet)$ is essential while defining mixed Hodge Module. However, for the vanishing theory studied in this paper, we only use its formal property: the strictness of certain functors with respect to $F_\bullet$. Hence we mute those information for the ease of notation.
\begin{thm}\label{T: Saito vanishing}
Let $X$ be a complex projective manifold, with a reduced divisor
$E$ and a semi-ample line bundle $\cL$. Assume further
that $\cL(eE)$ is an ample line bundle, for some nonnegative integer $e$. Let $(\cM, F_\bullet)$ be a pair underlying a mixed Hodge Module as described above. Then, we have the following vanishings:
\begin{align*}
    \H^i(\DR_X \Gr^F_\bullet\cM[*E]\otimes \cL)&=0, \text{for } i>0;\\
    \H^i(\DR_X \Gr^F_\bullet\cM[!E]\otimes \cL^{-1})&=0, \text{for } i<0.
\end{align*}
\end{thm}
Here, we do not require $E$ being normal crossing. In the case with the absence of $E$, it is just Saito vanishing \cite[(2.g)]{Sa90}. Please refer \cite[Theorem 20]{W17b} for this general form, with the proof essentially follows from Saito's proof of Saito vanishing. Please also note that, in \textit{loc. cit.}, we use $\Sp$ to denote the de Rham functor on right modules. We use the notation $\DR$ here, to follow the convention in \cite{Sab05} and \cite{Moc15}, and it has been shifted accordingly.

In late 80s', based on the works of K. Corlette, S. Donaldson and N. Hitchin, the nonabelian Hodge Theory is mainly developed by C. Simpson. As the starting point, Simpson studies the harmonic bundles and uses that to give extra structures on cohomologies of a compact K\"ahlar manifold, with the coefficient being a semisimple local system, instead of the constant local system as studied in the classical Hodge theory. Simpson also proposes the Meta theorem which states that \emph{what works for classical Hodge theory shall still work for nonabelian Hodge theory}. Later on, following Saito's idea of constructing mixed Hodge Modules, the theory of mixed twistor $\sD$-modules is initiated by C. Sabbah. Then, T. Mochizuki completes this spectacular theory in the past decade, which can also be viewed as a satisfying answer towards Simpson's Meta Theorem. In the same vein, heuristically, the vanishing theorem above can be generalized to the corresponding mixed twistor $\sD$-modules setting. 

The nonabelian Hodge Theory can be briefly viewed as building connections between certain topological objects, (e.g. semisimple local system,) analytic objects, (e.g. harmonic bundles,) and algebraic objects (, e.g. polystable Higgs bundle). In the topological setting, we have the Artin Vanishing, which we will use in the following form. For any perverse sheaf $P$ on a smooth affine variety $A$, we have 
\begin{align*}
    \H^i(A, P)=0, \text{for } i>0,\\
    \H_c^i(A, P)=0, \text{for } i<0.
\end{align*}
Then, follow the general guideline of the nonabelian Hodge Theory, we should expect getting corresponding vanishings, in the algebro-geometric setting. Following Saito's approach in \cite[\S 2.g]{Sa90}, we propose that the Kodaira-type vanishings are exactly such incarnations. Before we state the main vanishing results of this paper, let's set up the notion for mixed twistor $\sD$-modules. Since the author has very limited knowledge about wild mixed twistor $\sD$-modules, we will only focus on the regular case. 

For any complex manifold $X$, we denote $\cX=X\times \C_\lambda$, where $\C_\lambda$ is the affine line with $\lambda$ as its coordinate, and $p:\cX\to X,$ $q:\cX\to \C_\lambda$ the natural projections. Following the notations in \cite{Sab05} and \cite{Moc15}, in this paper, we say that a right algebraic $\sR_{X}$-module $\cM$ on $\cX$ \emph{underlies a mixed twistor $\sD$-module}, if there exists \emph{an algebraic graded polarizable regular mixed twistor $\sD$-module} represented by a filtered $\sR$-triple $(\cT, W)$, with the triple $\cT=(\cM', \cM'', C),$ and $\cM=\cM''$. We also use $\Xi_{\DR, \alpha}\cM$, for $\alpha\in \C^*$, to denote the coherent $\sD_X$-modules corresponding to $\cM/(\lambda-\alpha)\cM$ and $\Xi_{\Dol}\cM:=\cM/\lambda\cM$, which are exact functors from strict coherent $\sR_X$-modules to coherent $\sD_X$-modules and coherent $\sA_X:=\Gr^F(\sD_X)$-modules, respectively \cite[1.1.a, Definition 1.2.1]{Sab05}. As we remarked in the mixed Hodge Module case, the other data while defining a graded polarizable mixed twistor $\sD$-module is important. However, for the vanishing theorem, we only use the holomorphic picture of a mixed twistor $\sD$-module, and its formal property like the strictness of certain natural functors, so we will only keep track the $\sR_X$-module $\cM=\cM''$. Please also note that, although we may allow \emph{wild} mixed twistor $\sD$-modules, and it seems to the author that the same proof still holds, we certainly have not exploited the extra structures coming with them, and may deserve further study.

The following vanishing result is a direct nonabelian Hodge theoretic generalization of Theorem \ref{T: Saito vanishing}. 
\begin{thm}\label{T:twistor vanishing}
Let $X$ be a complex projective manifold, with a reduced divisor
$E$ and a semi-ample line bundle $\cL$. Assume further
that $\cL(dE)$ is an ample line bundle, for some nonnegative integer $d$. Let $\cM$ be a right $\sR_X$-module underlying a mixed twistor $\sD$-module. Then, we have 
the following vanishings:
\begin{align*}
    \H^i(\DR_X \Xi_{\Dol}\cM[*E]\otimes \cL)&=0, \text{for } i>0;\\
    \H^i(\DR_X \Xi_{\Dol}\cM[!E]\otimes \cL^{-1})&=0, \text{for } i<0.
\end{align*}
\end{thm}

In the previous theorem, $\cM[*E]$ and $\cM[!E]$ are prolongations of $\cM$ as coherent $\sR_{X}$-modules as in \cite[3.1.2, 3.1.3]{Moc15}. We will use $\cM(*E)$ to denote the canonical prolongation as a coherent $\sR_{X(*E)}$-module \cite[3.1.1]{Moc15}. 

Naturally, we want to give the previous vanishing a Kawamata-Viehweg Vanishing type formulation, as in Theorem \ref{T: Nakano Vanishing}. The next vanishing theorem can be viewed as a such generalization, and it is new even in the setting of mixed Hodge Modules. Such generalization is critical when we need to apply it on certain birational model, e.g. the application of such vanishing results on studying toric singularities, \cite{W24}. The formulation is quite technical, and we will define the notations in \S3. At this stage, for reader's convenience, we introduce the notation
$$\mathbf{V}^{(\alpha)D}_{\mathbf{a}}\mathcal{M} (*D)=\cap V^{(\alpha)D_i}_{a_i} \mathcal{M} (*D),$$
where $V^{(\alpha)D_i}_{\bullet} \mathcal{M} (*D)$ is the KM-filtration of $\mathcal{M}(*D)$ locally around $\lambda=\alpha$. 
We have $\mathbf{V}_{\mathbf{0}}^{D}\sR_{X(*D)}$ is a coherent sub-ring of $\sR_X$, denoted by $\sR_{(X,D)}$. We set $\sA_{(X,D)}:=\sR_{(X,D)}/\lambda\sR_{(X,D)}$  a coherent sub-ring of $\sA_X:=\Gr^F(\sD_X)$. We have $\mathbf{V}^{(\alpha)D}_{\mathbf{a}}\mathcal{M} (*D)$ is naturally a $\sR_{(X,D)}^{(\alpha)}$-module, and $\Xi_{\Dol}\mathbf{V}^{(0)D}_{\mathbf{a}}\mathcal{M} (*D)$ is naturally an $\sA_{(X,D)}$-module.

Please note that the multi-indexed filtration $\mathbf{V}^{(\alpha)D}_{\bullet}$ does not behave well in general. For example, $\mathbf{V}^{(\alpha)D}_{\mathbf{a}}\mathcal{M} (*D)$ may not be a coherent $\sR_{(X,D)}$-module. We need the $V$-compatibility on $\cM(*D)$, Definition \ref{D:V strict}.
\begin{thm}\label{T:KV vanishing}
Let $X$ be a smooth projective variety, with a reduced, normal crossing divisor $D$. Let $\cM$ be a right $\sR_X$-module underlying a mixed twistor $\sD$-module, and assume that $\cM(*D)$ is \emph{$V^{(0)}$-compatible with respect to $D$}. (In particular, it is the case when $\cM$ corresponds to a tame harmonic bundle on $X\setminus D$). Separate $D$ into two groups of components: $B+C=D$, and assume that we have a semi-ample divisor $A$, such that $A+\be C$ is ample for some $\be\geq \0$. Let $L$ be a divisor, such that
$$L \equiv_{\num} aA+\bb B +\bc C,$$
with $a>0$.
Then, we have the following vanishings:
\begin{align*}
    \H^i(\DR_{(X,D)}(\Xi_{\Dol}\bV^{(0)B}_{<-\bb} \bV^{(0)C}_{-\bc} \cM (*D))\otimes \cO_X(L))=0&, \text{for } i>0;\\
    \H^i(\DR_{(X,D)}(\Xi_{\Dol}\bV^{(0)B}_{\bb} \bV^{(0)C}_{<\bc} \cM (*D))\otimes \cO_X(-L))=0&, \text{for } i<0.
\end{align*}
\end{thm}

\begin{remark}\label{R:shift V fil}
If $\cM(*D)$ is
$V^{(0)}$-compatible, then for every $\bk\in\Z^n$,
\[
\bV^{(0)D}_{\bd+\bk}\cM(*D)
=
\bV^{(0)D}_{\bd}\cM(*D)
\otimes p^*\cO_X(\bk D).
\]
\end{remark}

\begin{remark}
Actually, we only need to show the case that $\be=\0$, i.e. $A$ itself is ample, since we always have
$$L \equiv_{\num} aA+\epsilon\be C +\bb B +(\bc-\epsilon\be) C,$$
noting that $ aA+\epsilon\be C $ is ample, and $\bV^{(0)C}_{-\bc}=\bV^{(0)C}_{-\bc+\epsilon\be}$, $\bV^{(0)C}_{<\bc}=\bV^{(0)C}_{<\bc-\epsilon\be}$, for some $0<\epsilon\ll1$, by the right-semicontinuity
of $V_\bullet$ and the left-semicontinuity of
$V_{<\bullet}$.

The reason that we keep the superficially more general form as stated above, is due to that it actually carries natural geometric information that will be clarified in the proof. Furthermore, following the same proof as in \S5, Theorem \ref{T:KV vanishing} can be further generalized to the form that assume $A+\be (C+E)$ is ample for some effective divisor $E$, and replace $\cM(*D)$ in the first (resp. second) vanishing by $\cM[*E](*D)$ (resp. $\cM[!E](*D)$), as in Theorem \ref{T:twistor vanishing}. 
\end{remark}

\begin{remark}\label{R: R div}
The above theorem also works for $\R$-divisors by the following reduction. Due to the previous remark, we can assume $A$ is ample. Hence, we can find $\0<\bb''\ll \1$, and  $\0<\bc''\ll \1$, so that $aA-\bb''B+\bc''C$ is still ample, and both $\bb':=\bb+\bb''$ and $\bc':=\bc-\bc''$ are rational. Hence, we have
$$L \equiv_{\num} (aA-\bb''B+\bc''C) +\bb' B +\bc' C.$$
We note that 
$$aA-\bb''B+\bc''C\equiv_{\num}
L-\bb'B-\bc'C,$$ 
with the right hand side is a rational numerical class. Since it is
ample, it can be represented by an ample $\Q$-divisor. By the one-sided semicontinuity and local discreteness
of the KM-filtrations, we may choose $\bb''$ and $\bc''$
sufficiently small so that
\[
\bV^B_{<-\bb}\bV^C_{-\bc}
=
\bV^B_{<-\bb'}\bV^C_{-\bc'}
\]
and
\[
\bV^B_{\bb}\bV^C_{<\bc}
=
\bV^B_{\bb'}\bV^C_{<\bc'}.
\] From now on, we only consider $\Q$-divisors, unless explicitly stated otherwise.
\end{remark}

We will also establish the relative versions of the previous two vanishing theorems in \S 2. Actually, it is not straight-forward to get such generalization as e.g. \cite[Theorem 1-2-3]{KMM}, using a log-smooth compactification and Serre vanishing on coherent $\cO$-modules. This is because, in our case, we are not working in the derived category of coherent sheaves. However, due to our proof of Theorem \ref{T:twistor vanishing} being functorial, we can also get a functorial proof of its relative version, Theorem \ref{T:relative twistor vanishing}, which does not use a log-smooth compactification. We then show a Nadel-type vanishing with multiplier-ideals from $\Q$-divisors, and state an effective global generalization theorem, as an application. Moreover, we will mention another geometric application of the vanishing results in this paper, for studying toric singularities.

It is also natural to generalize the injectivity/surjectivity results in \cite[\S 5]{EV92}, \cite{Wu17}, \cite{Wu21}, to the nonabelian Hodge setting. However, the author finds such type of results have a different nature of approach, and they cannot directly imply those vanishing results above using the method summarized in  \cite[\S 1, 2. proof]{EV92}, with the help of Serre Vanishing, for the same reason in the relative vanishing case. Hence, the author decides to leave it for another occasion.

We now show how the two vanishing theorems discussed above cover nearly all known instances of Kodaira-type vanishing in algebraic geometry. Nevertheless, the author has no intention to say that we have new proofs to those brilliant results. On the contrary, the vanishing theorems above shall be taken as a natural summary of currently known Kodaira-type vanishings, and the proofs are based on the proofs of those results. The proofs mainly adopt Esnault and Viehweg's geometric construction \cite{EV92}, Saito's proof of Saito Vanishing, and the theory of mixed twistor $\sD$-modules, due to Sabbah and Mochizuki. However, it seems hopeless to the author that the vanishings in this paper can fully recover Nadel vanishing in the analytic setting, especially, those vanishings using Nakano semipositive vector bundles, that is proved using $L^2$ method, e.g. \cite{deC98a}, \cite{deC98b}. This is due to that it is not guaranteed that we can construct a mixed twistor $\sD$-module, which naturally associates to a specific Nakano semipositive vector bundle.

We first restrict ourselves to the case that $\cM$ being the associated Rees object, \cite[0.2]{Sab05}, of a mixed Hodge Module $(\cN, F_\bullet)$, with a right $\sD$-module $\cN$, and an increasing Hodge filtration $F_\bullet$. We have that $\Xi_{\Dol}\cM:=\cM/\lambda\cM$ is just taking the associated graded pieces, $\Gr^F_\bullet \cN$. Then, Theorem \ref{T:twistor vanishing} in this case is just Theorem \ref{T: Saito vanishing}.

For Theorem \ref{T:KV vanishing}, if we only consider $\cM$ underlies a graded polarizable variation of mixed Hodge structures on $X$, $C=0$, and $\0\leq \bb<\1,$ we have 
$$V^B_{<-\bb} \cM(*D)\simeq \cM.$$
In particular,  if we take $\cM\simeq \omega_\sX$, i.e. it corresponds to the trivial variation, we get the Kawamata-Viehweg vanishing for klt pairs. Hence, via a standard argument, e.g. Corollary \ref{C:Nadel V}, it implies the big and nef vanishing for the lowest graded piece of a mixed Hodge Module, which has been proved by Suh \cite{Suh18} and Wu \cite{Wu17} independently. Let's also remark that the $C$ part can be useful in case we want to deal with log-canonical singularities, e.g. the application on showing the zero locus of holomorphic log-one forms, \cite{W17b}.

If we have a reduced, possibly singular Cartier divisor $D$, by taking a log resolution $f:(X', D')\to (X, D)$, with $f^*dD=\bd' D'$. We have the multiplier ideal
\begin{equation}\label{E:from V to J }
    \cJ(dD)\otimes \omega_X\simeq f_*\Xi_{\Dol}\bV^{D'}_{<-\bd'}\omega_{\cX'}(*D').
\end{equation}
Hence we get the Nadel vanishing with the multiplier ideal of $\Q$-divisors. Such relation between multiplier ideals and Hodge modules has also been studied in \cite{BS05} and \cite{MP20}. See also Corollary \ref{C:Nadel V} in the next section.

In the case that $\cM$ underlies a tame harmonic bundle on $X\setminus D$, this is equivalent to that $\Xi_{\Dol} \cM$ gives us a slope parabolic polystable Higgs bundle on $(X, D)$, due to \cite[Theorem 1.4]{Moc06}. Since a parabolic semistable Higgs bundle can be realized as a sequentially extension of stable Higgs bundles, as argued in \cite[Lemma 7.1]{AHL}, then the two vanishings above generalize those vanishing results proved by Arapura-Hao-Li \cite{AHL}, and by Deng-Hao \cite{DH}. In \cite{DH}, their vanishing is more refined in the sense that they build a relation between the range of the vanishing degrees and the number of positive eigenvalues of the curvature form of $\cL$. In the case that $\cL$ being $k$-ample as in \cite[\S 1]{Som78}, (which is a stronger condition, see also \cite{Tot13},) we have a projective dominate morphism $f:X\to Y$, with $\cL^n=f^*\cA$ for some ample line bundle $\cA$ on $Y$, and $n\in \Z^+$. Hence we can achieve the vanishing with the expected range as in \cite{DH}, due to the compatibility of the direct image, $\DR$ and $\Xi_{\Dol}$ functor. We leave the details to the interested readers. 

In \S3, we recall the notion of multi-indexed Kashiwara-Malgrange (KM) filtration and its compatibility, and we use it to get the logarithmic comparison in the setting of $\sR$-modules as in \cite{W17a}. In \S4, we recall the geometric construction in \cite{EV92}, which is constructing a cyclic cover from a normal crossing $\Q$-divisor, and study the induced local systems and their multi-indexed KM filtration. We give the proofs of our main vanishing results in \S5.

In this paper, we use $\otimes$ (without sub-index) to denote the tensor product over the corresponding structure sheaf $\cO$, unless specified otherwise. We use $f_+$ and $f_\dagger$ for the derived direct-image
functors of $\sR$-modules, and $\cH^if_+$ and
$\cH^if_\dagger$ for their cohomology modules. For sheaves, we
use $Rf_*$ and $Rf_!$ for the derived functors and $R^if_*$
and $R^if_!$ for their cohomology sheaves. We use $\H^\bullet$ (resp. $\H^\bullet_c$) to denote taking the hypercohomoloy (resp. hypercohomoloy with proper support), which is the same as the (derived) functor of taking derived direct image $a_*$ (resp. $ a_!$), where $a$ is the canonical map to $\text{Spec}(\C)$. We consider \emph{right} $\sD$-modules or $\sR$-modules by default, unless specified otherwise. Although we only consider the vanishing theory for algebraic mixed twistor $\sD$-modules on a quasi-projective variety, the constructions in \S3 and \S4 still work in the analytic setting.

\begin{ack}
    The author would like to thank Takuro Mochizuki and Christian Schnell for their patience and kindness for clarifying several mistakes appeared in an early draft. The author also feels grateful to Jiyuan Han for teaching him vanishings in the analytic setting during the preparation of this paper, and to Claude Sabbah, Ruijie Yang for reading this paper and useful suggestions. The author is also thankful to an anonymous referee for finding the correct statement for the relative vanishings. During the preparation of this paper, the author is supported by the National Key R\&D Program of China (\#2024YFA1014400).
\end{ack}

\section{Relative Vanishings and applications}
In this section, we state the relative vanishing, a generalized nef and big vanishing, and an effective global generalization theorem, as an application. Most notations shall be standard, following \cite{Sab05} and \cite{Moc15}. Some notations will be carefully defined in the next section.

Let's first state a relative version of Theorem \ref{T:twistor vanishing}.
\begin{thm}\label{T:relative twistor vanishing}
Let $f:X\to S$ be a projective morphism between smooth quasi-projective varieties. On $X$, assume that we have a reduced divisor $E$ and an $f$-semi-ample line bundle $\cL$. Assume further that $\cL(dE)$ is an $f$-ample line bundle, for some nonnegative integer $d$. Let $\cM$ be a right $\sR_X$-module underlying a  mixed twistor $\sD$-module on $X$. Then, 
we have the following vanishings:
\begin{align*}
    R^if_*((\Xi_{\Dol}\cM[*E]\otimes^L_{\sA_X} f^*\sA_S) \otimes \cL)&=0, \text{for } i>0;\\
    R^if_*((\Xi_{\Dol}\cM[!E]\otimes^L_{\sA_X} f^*\sA_S)\otimes \cL^{-1})&=0, \text{for } i<0.
\end{align*}
\end{thm}

We can also show a relative version of Theorem \ref{T:KV vanishing}:
\begin{thm}\label{T:relative KV vanishing}
Fix a projective morphism between smooth quasi-projective varieties $f:X \to S$.
$D$ is a reduced, normal crossing divisor on $X$. Let $\cM$ be a right $\sR_X$-module underlying \emph{a mixed twistor $\sD$-module} on $X$, and assume that $\cM(*D)$ is \emph{$V^{(0)}$-compatible with respect to $D$}. Separate $D$ into two groups of components: $B+C=D$, and assume that we have an $f$-semi-ample divisor $A$, such that $A+\be C$ is $f$-ample for some $\be\geq \0$. Let $L$ be a divisor, such that
$$L \sim_{\Q} aA+\bb B +\bc C,$$
with $a>0$.
Then, we have the following vanishings:
\begin{align*}
    R^if_*((\Xi_{\Dol}\bV^{(0)B}_{<-\bb} \bV^{(0)C}_{-\bc} \cM (*D)\otimes^L_{\sA_{(X,D)}} f^*\sA_S)\otimes \cO_X(L))=0&, \text{for } i>0;\\
    R^if_*((\Xi_{\Dol}\bV^{(0)B}_{\bb} \bV^{(0)C}_{<\bc} \cM (*D)\otimes^L_{\sA_{(X,D)}} f^*\sA_S)\otimes \cO_X(-L))=0&, \text{for } i<0.
\end{align*}
\end{thm}

\begin{remark}
    For any $\sA_{(X,D)}$-module $\cN$ and $f:X\to S$ as above, the natural direct image functor is defined by $R^\bullet f_*(\cN\otimes^L_{\sA_{(X,D)}} f^*\sA_S )$, and it is compatible with the direct image functor $f_\dagger$ in the setting of mixed twistor $\sD$-modules. 
    
    When $\cM$ underlies a mixed Hodge module, the lowest
Hodge-graded piece is the lowest internal-degree summand of the
logarithmic Spencer complex. Since the positive-degree symbols
act trivially on this piece and the differentials preserve the
internal grading, it is a direct summand of
\[
\Xi_{\Dol}\cM
\otimes_{\sA_{(X,D)}}^Lf^*\sA_S
\]
as a complex of quasi-coherent sheaves. In particular, by considering $\cM=\omega_\sX$, $f$-ample $\Q$-divisor $A$, $\bc C =\lceil A\rceil-A $, $L:=A+\bc C=\lceil A\rceil$, we have $\Xi_{\Dol} \bV^{(0)C}_{-\bc} \cM (*C)\otimes \cO(L)=\omega_X(\lceil A\rceil)$. Here, we assume $S$ is smooth. If $S$ is singular, we can embed $S$ into a smooth variety $Y$ and consider the induced map $X\to Y$, since we do not need our $f$ being surjective. This implies \cite[Theorem 1-2-3]{KMM} for the $A$ ($D$ in the literature) being $f$-ample case. The more general case assuming $A$ being $f$-nef and $f$-big can be reduced to the previous case by applying a similar argument later in Corollary \ref{C:Nadel V}. 
\end{remark}

Then, we can show the following Nadel vanishing on a numerically flat vector bundle with vanishing Chern classes.

\begin{coro}\label{C:Nadel V}
Let $X$ be a smooth projective variety, and let $\cN$ be a
numerically flat vector bundle on $X$. Equivalently, $\cN$ is a vector bundle with vanishing rational Chern classes which is slope-semistable with respect to one,
and hence every, ample polarization, \cite{DPS94}. Given a divisor $L$, such that 
$$L\equiv_{\num} aA+bB, a>0, b\geq0$$
with $A$ a nef and big divisor, and $B$ an effective divisor. Then we have
\begin{equation*}
    \H^i(\cN\otimes \omega_X(L)\otimes \cJ(bB))=0, \text{for } i>0.
\end{equation*}
where $\cJ(bB)$ stands for the multiplier ideal of the $\Q$-divisor $bB$.
\end{coro}
$\cJ(bB)$ can also be replaced by $\cJ(b|B|)$ the multiplier ideal of linear series, or $\cJ(||bB||)$ the asymptotic multiplier ideal. Let's refer \cite{Laz04II} for more details. 
\begin{proof}
By \cite{DPS94}, $\cN$ admits a filtration by vector
subbundles whose successive quotients are unitary flat.
Since the desired vanishing is preserved under extensions, it
is enough to prove it when $\cN$ is unitary flat.

Take an embedded log-resolution of $X$ and the support of $B$, getting $\pi:X'\to X.$ Set $\pi^*aA=a'A'$ and $\pi^*bB=\bb'B'$. Since $A$ is nef and big, so is $a'A'$. Up to a further log-resolution, we have $a'A'\sim_{\Q}a^\circ A^\circ +\be E,$ with $a^\circ A^\circ$ ample and $\0< \be \ll \1$. (Since we can apply $$a'A'\sim_{\Q}(\frac{a^\circ}{m} A^\circ+\frac{a'(m-1)}{m}A') +\frac{1}{m}\be E,$$
with the $\Q$-divisor in the parentheses being ample.)
To summarize, we have 
$$\pi^*L\equiv_{\num}a^\circ A^\circ +\bb'B'+\be E.$$
Set $\cN'=\pi^*\cN$, and it, equipped with the trivial higgs map, still can be lifted as a variation of twistor structures, and we use $\cM'$ to denote its corresponding $\sR$-module. In particular, it is non-characteristic with respect to any smooth divisor, hence the multi-indexed KM-filtration of $\cM'(*(B'+E))$ only jumps at the integers, along each component of $B'$ and $E$, e.g. (\ref{E: multi KM fil of * ext}). 
In the case that $B'$ and $E$ share no common component, we have 
\begin{equation}\label{E:small e V-fil}
    \bV^{B'}_{<-\bb'}\bV^E_{<-\be} \cM'(*(B'+E))\simeq \bV^{B'}_{<-\bb'}\bV^E_{<\0} \cM'(*(B'+E)) \simeq \bV^{B'}_{<-\bb'}\cM'(*B'),
\end{equation}
where the second identity can be checked directly. 
In the case that $B'$ and $E$ share common components, e.g. $B'_0=E_0$, we change the first term by omitting $e_0E_0$ in $E$, and use $b'_0+e_0$ to replace $b'_0$, both as index of KM-filtration and coefficient of $B'_0$. Since $e_0$ is small, it is still isomorphic to the other two terms with index $\bb'$ the original coefficients of $B'$. See also the remark below. Now, due to Theorem \ref{T:KV vanishing}, we have 
$$\H^i(\DR_{(X', B'+E)}(\Xi_{\Dol}\bV^{B'}_{<-\bb'}\bV^E_{<-\be} \cM'(*(B'+E)))\otimes \pi^*\cO(L))=0, i>0.$$
Since we also have $\Xi_{\Dol}\bV^{B'}_{<-\bb'}\cM'(*B')\simeq \cN'\otimes \Xi_{\Dol}\bV^{B'}_{<-\bb'}\omega_{\cX'}(*B')$, and the Higgs connection has set to be trivial, combining (\ref{E:small e V-fil}),
$$\DR_{(X', B'+E)}(\Xi_{\Dol}\bV^{B'}_{<-\bb'}\bV^E_{<-\be} \cM'(*(B'+E)))$$
decomposes as $\bigoplus (\cN'\otimes \wedge^i\cT_{(X', B'+E)} \otimes \Xi_{\Dol}\bV^{B'}_{<-\bb'}\omega_{\cX'}(*B'))[i]$

Due to local vanishing, we have
$$R \pi_*\Xi_{\Dol}\bV^{B'}_{<-\bb'}\omega_{\cX'}(*B')\simeq R \pi_*\omega_{X'}(-\lfloor\bb'\rfloor B')\simeq R^0 \pi_*\omega_{X'}(-\lfloor\bb'\rfloor B') \simeq \omega_{X} \otimes \cJ(bB),$$
See also \cite{BS05}. (Of course, we can directly apply Theorem \ref{T:relative twistor vanishing}, but it is not necessary.)
Combining the vanishing above and the projection formula, we get the vanishing we need. 
\end{proof}
\begin{remark}
Note that, for the nef and big vanishing, we cannot add the $C$ part in Theorem \ref{T:KV vanishing} in a naive way. The reason is that we cannot identify $\bV^{E}_{\0}$ and $\bV^{E}_{-\be}$ as in (\ref{E:small e V-fil}), even when $\be$ is very small.
\end{remark}
\begin{remark}
The essential point of this version of Nadel vanishing is Hitchin-Kobayashi correspondence. Once we assume that, we can view a stable vector bundle $\cE$ as a variation of Hodge structures coming from unitary representations, i.e. with trivial Hodge filtration. Hence itself is the lowest filtered piece, and we can apply the vanishing in the setting of Hodge module, e.g. \cite{Wu17}. Another way to achieve the vanishing is to use the fact that, each unitary-flat stable quotient has zero curvature and is therefore Nakano semipositive. We can apply the usual $L^2$ type Nadel vanishing. 
\end{remark}

It is natural to apply the previous Nadel vanishing to get the following effective global generation result, and we follow the exposition in \cite[10.4]{Laz04II}. Please also refer \cite{deC98a}, \cite{deC98b} for a much more comprehensive study on this topic.
\begin{thm}[Theorem of Angehrn and Siu]\label{T:AS thm}
Let X be a smooth projective variety of dimension $n$, with $A$ an ample divisor on it. Fix a point $x\in X$, and assume that, for every irreducible subvariety $Z\subset X$ passing through $x$, (including $X$ itself,) setting $m$ as its dimension, and assume that
$$(A^m\cdot Z)>(\frac{1}{2}n(n+1))^m.$$
Then, for any numerically flat vector bundle $\cN$ on $X$, $\cN\otimes \omega_X(A)$ is free at $x$, i.e. the evaluation map of the global sections of $\cN\otimes \omega_X(A)$ at $x$ is surjective. 

In particular, if $A\equiv_{\num}kL,$ for some $k> \frac{1}{2}n(n+1)$ and an ample divisor $L$, then $\cN\otimes \omega_X(A)$ is free, i.e. globally generated.
\end{thm}

\begin{proof}
According to \cite[10.4.C]{Laz04I}, we are able to find an effective $\Q$-divisor $dD$ on $X$, with $lct(D;x)=d$, and $x$ is an isolated component of the non-klt locus, and $dD\equiv_{\num} \lambda A$, for some $\lambda<1$. Due to the vanishing in Corollary \ref{C:Nadel V}, we have
$$H^1(X, \cN\otimes \omega_X(A)\otimes \cJ(dD))=0.
$$
As argued in \cite[10.4.C]{Laz04I}, we have the surjection of the following evaluation map at $x$: 
$$H^0(X, \cN\otimes \omega_X(A))\to H^0(X, \cN\otimes \omega_X(A)\otimes \C(x)),
$$
where $\C(x)$ denoting the one-dimensional sky-scraper sheaf, supported at $x$.
\end{proof}

We also briefly mention another geometric application of the Kodaira-type vanishings on studying toric singularities. 

\begin{definition}
Let $X$ be a normal toric variety. Let $B$ and $C$ be two reduced torus-invariant Weil divisors.  We also require $B$ and $C$ share no common irreducible components. We use $(X, !B, *C)$ to denote the data and call it \emph{a toric triple}.

We say $(X, !B, *C)$ is a toroidal triple, if for any point $x\in X$, analytically locally, it is isomorphic to a toric triple $(T_x, !B_x, *C_x)$.

Given a toroidal triple $(X, !B, *C)$, we define the sheaf of Danilov's log-differential $p$-forms $\Omega^p_{(X,!B,*C)}$ on $X$, as $(\Omega^p_X(\log (B+C))(-B))^{\vv}$, the reflexive sheaf that corresponds to the counterpart on the smooth locus of $X$. 
\end{definition}

In \cite{W24}, based on the arguments in \S3 and \S4, and extra combinatorial constructions, we can deduce the following generalized Kawamata-Viehweg-Bott Vanishing.
\begin{thm}\label{T: KVB vanishing of toroidal tri.}
    Let $X$ be a proper normal toroidal variety, with $\dim X=n,$ and $(X, !B, *C)$ being a toroidal triple. For any Weil divisor $L$ on $X$ that is $\Q$-linear equivalent to $A+\bb B-\bc C,$
    where $A$ is an ample $\Q$-divisor, with all entries of $\bb$ and $\bc$ in $[0, 1]$, then we have
    \begin{equation*}
        H^q(X, (\Omega^p_{(X, !B, *C)}\otimes \cO_X(L))^{\vv})=0, \text{for } p+q>n.
    \end{equation*}
    On the other hand, if $L$ is $\Q$-linear equivalent to $-A+\bb B-\bc C,$ we have
    \begin{equation*}
        H^{n-q}(X, (\Omega^{n-p}_{(X, !B, *C)}\otimes \cO_X(L))^{\vv})=0, \text{for } p+q>n.
    \end{equation*}
Moreover, if $(X, !B, *C)$ is a toric triple, then the both vanishings above hold for all $p\geq 0$ and $q>0$.
\end{thm}
Using such vanishings, we can establish Deligne's logarithmic comparison theorem and the $E_1$-degeneration of the corresponding Hodge-de Rham spectral sequence within the context of toroidal embeddings. Please refer to \emph{loc. cit.} for the precise statements and proofs.

\section{Multi-indexed KM-filtration and Logarithmic comparison}
In this section, we generalize some results in \cite{W17a} about logarithmic comparison in the mixed Hodge module setting to the mixed twistor $\sD$-module case.

Let us first recall the definition of the Kashiwara--Malgrange (KM-)filtration on a coherent $\sR_{X(*D)}$-module, with respect to a smooth component $H$ of a reduced simple normal crossing divisor $D$ on a complex manifold $X$. Since, in this paper, given an $\sR$-module $\cM$ on $X$, we will find ourselves only using its information on the open part $X\setminus D$, it is more natural to consider the KM-filtration on the coherent $\sR_{X(*D)}$-module
\[
\cM(*D):=\cM\otimes_{\cO_{\cX}}\cO_{\cX}(*(p^*D))
\]
than to consider the KM-filtration on coherent $\sR_X$-modules such as $\cM[*D]$ or $\cM[!D]$; see also Remark \ref{R:shift V fil}. Actually, $\cM[*D]$ and $\cM[!D]$ themselves are built from the KM-filtration on $\cM(*D)$, as in \cite[\S 3.1]{Moc15}.

There is a canonical $\Z$-indexed increasing filtration on $\sR_{X(*D)}$, stalk-wise defined, for any $(x,\alpha)\in\cX$, by
\begin{align*}
V^H_k\sR_{X(*D),(x,\alpha)}
:=\bigl\{P\in\sR_{X(*D),(x,\alpha)}\mathrel{\big|}\;&
P\cdot p^*\cO_X(iH)_{(x,\alpha)}\\
&\subset p^*\cO_X((i+k)H)_{(x,\alpha)},\ \forall i\in\Z\bigr\}.
\end{align*}
We note that $V^H_0\sR_{X(*D)}$ is a coherent ring. It can be directly checked that
\[
V^H_k\sR_{X(*D)}=V^H_0\sR_{X(*D)}\otimes\cO_X(kH),
\]
and $V^H_\bullet\sR_{X(*D)}$ is a good filtration, \cite[Appendix III]{Bj93}, as a $V^H_0\sR_{X(*D)}$-module. For $a\in\R$, we use the convention
\[
V^H_a\sR_{X(*D)}:=V^H_{\lfloor a\rfloor}\sR_{X(*D)}.
\]

Now, we use $\cN$ to denote a strict coherent $\sR_{X(*D)}$-module. For any $\alpha\in\C_\lambda$, we use $\cN^{(\alpha)}$ to denote the restriction of $\cN$ to an open subset
\[
\cX^{(\alpha)}:=X\times\Delta(\alpha,\epsilon)\subset\cX,
\]
for some sufficiently small $\epsilon>0$, where $\Delta(\alpha,\epsilon)$ is the open disk in $\C_\lambda$ centered at $\lambda=\alpha$.

For $u=(a,\beta)\in\R\times\C$, we recall the functions
\begin{equation}\label{E:p and e}
\fp(\lambda,u)=a+2\operatorname{Re}(\lambda\overline\beta),
\qquad
\fe(\lambda,u)=\beta-a\lambda-\overline\beta\lambda^2.
\end{equation}

\begin{definition}\label{D: KM filtration}
Let $\cN$ be a strict coherent $\sR_{X(*D)}$-module. We say that it is \emph{strictly specializable} along $H$, a smooth component of $D$, \cite[2.1.2.2]{Moc15}, if, for every $\alpha\in\C_\lambda$, there exists a \emph{Kashiwara--Malgrange (KM-)filtration} $V^{(\alpha)H}_\bullet$ on $\cN^{(\alpha)}$, for some sufficiently small $\epsilon>0$, which is an exhaustive $\R$-indexed increasing filtration by coherent $V^H_0\sR_{X(*D)}$-modules, satisfying the following conditions.
\begin{enumerate}
    \item The filtration is locally discrete. More precisely, after shrinking around any point $P\in\cX^{(\alpha)}$, its set of jumps in every bounded interval is finite. In particular, for every $a\in\R$, there exists some $\epsilon_a>0$ such that
    \[
    V^{(\alpha)H}_a\cN=V^{(\alpha)H}_{a+\epsilon_a}\cN.
    \]
    Here and below, we omit the superscript $(\alpha)$ on $\cN$ for ease of notation.

    \item Each $\sR_{H(*D'|_H)}$-module
    \[
    \Gr^{V^{(\alpha)H}}_a\cN
    :=V^{(\alpha)H}_a\cN/V^{(\alpha)H}_{<a}\cN,
    \qquad
    V^{(\alpha)H}_{<a}\cN:=\bigcup_{b<a}V^{(\alpha)H}_b\cN,
    \]
    is strict, where $D'=D-H$.

    \item Locally around any point $P\in\cX^{(\alpha)}$, if $t$ is a local holomorphic function defining $H$, then
    \[
    V^{(\alpha)H}_a\cN\cdot t=V^{(\alpha)H}_{a-1}\cN,
    \qquad
    V^{(\alpha)H}_a\cN\cdot\eth_t\subset V^{(\alpha)H}_{a+1}\cN,
    \]
    for all $a\in\R$. The first equality holds for all $a$ because $\cN$ is localized along $H$.

    \item For any $a\in\R$ and $P\in\cX^{(\alpha)}$, there exist a finite set
    \[
    \cK(a,\alpha,P)\subset
    \{u\in\R\times\C\mid\fp(\alpha,u)=a\}
    \]
    and an integer $N\gg0$ such that
    \begin{equation}\label{E:one KMS polynomial}
    \left(\prod_{u\in\cK(a,\alpha,P)}
    (t\eth_t+\fe(\lambda,u))\right)^N
    \end{equation}
    acts as zero on $\Gr^{V^{(\alpha)H}}_a\cN$. Equivalently, right multiplication by the operator in \eqref{E:one KMS polynomial} sends $V^{(\alpha)H}_a\cN$ into $V^{(\alpha)H}_{<a}\cN$.

\end{enumerate}
\end{definition}

\begin{remark}\label{R:alpha neighborhood}
The KM-filtration $V^{(\alpha)H}_\bullet$ on $\cN^{(\alpha)}$ is unique if it exists, \cite[Lemma 22.3.4]{Moc11}. Notice that it is a filtration on a neighborhood $\cX^{(\alpha)}$, and not only on the fiber $\lambda=\alpha$. Strict specializability means that such a filtration exists locally near every $\alpha\in\C_\lambda$. In general, proving the case $\alpha=0$ alone does not imply the result near an arbitrary $\alpha$.
\end{remark}

We now consider the multi-indexed KM-filtration.

\begin{remark}\label{R:NC local branches}
All constructions involving multi-indexed KM-filtrations are
understood analytically locally. On a normal crossing chart
$U\to X$, we write
\[
D|_U=H_1+\cdots+H_r,
\]
where the $H_\nu$ are the smooth local branches of $D$. All
multi-indices, intersection lattices, simultaneous graded
modules, and compatibility conditions below are defined with
respect to these local branches.

If several local branches lie over the same global irreducible
component of $D$, they are assigned the same coefficient and
belong to the same block $B$ or $C$. The compatibility
conditions are required for every subcollection of the local
branches. On overlaps, changes of normal crossing charts only
permute the branches; hence the intersection lattices and the
canonical morphisms below agree and descend to $X$.

Consequently, in the proofs we may pass to a normal crossing
chart and assume that the divisor is simple normal crossing.
\end{remark}

Let $D^+=D+E$ be a reduced normal crossing divisor, where
\[
D=D_1+\cdots+D_n
\]
is a union of irreducible components and $E$ has no common component with $D$. For ease of later usage, when we need to apply induction, we separate $D$ into two reduced normal crossing divisors $D=B+C$, with
\[
B=B_1+\cdots+B_m.
\]
For any $\bk=(k_1,\ldots,k_m)\in\Z^m$, we denote
\begin{equation}\label{E: def v fil}
\mathbf V^B_{\bk}\sR_{X(*D^+)}
:=\bigcap_{i=1}^m V^{B_i}_{k_i}\sR_{X(*D^+)}.
\end{equation}
For $\bb\in\R^m$, we set $\mathbf V^B_{\bb}\sR_{X(*D^+)}:=\mathbf V^B_{\lfloor\bb\rfloor}\sR_{X(*D^+)}$.

For $\bk=\mathbf0$, $\mathbf V^B_{\mathbf0}\sR_{X(*D^+)}$ is a coherent sub-ring of $\sR_{X(*(C+E))}$, denoted by
\[
\sR_{(X(*(C+E)),B)}.
\]
In particular, if $D^+=D$ and $B=D$, we denote it by $\sR_{(X,D)}$. We set
\[
\sA_{(X(*(C+E)),B)}:=\sR_{(X(*(C+E)),B)}/\lambda\sR_{(X(*(C+E)),B)},
\]
a coherent sub-ring of $\sA_{X(*(C+E))}:=\sR_{X(*(C+E))}/\lambda\sR_{X(*(C+E))}$.

\begin{remark}\label{R: KM fil on log R mod}
Let $\cN$ be a coherent $\sR_{(X(*(C+E)),B)}$-module, and let $H$ be a smooth component of $D^+-B$. Working locally, the coordinate functions defining $B$ and their compatible differential operators commute with $t$ and $\eth_t$, where $t$ defines $H$. Hence, the notion of being strictly specializable along $H$ and the corresponding KM-filtration in Definition \ref{D: KM filtration} generalize to this setting, replacing $\sR_{X(*D)}$ by $\sR_{(X(*(C+E)),B)}$ and replacing the ring on the graded module by the corresponding logarithmic ring on $H$.
\end{remark}

Let $\cN$ be a strict coherent $\sR_{X(*D^+)}$-module which is strictly specializable with respect to all $D_i$. This is the case, in particular, when $\cN=\cM(*D^+)$ and $\cM$ is the underlying $\sR_X$-module of a good mixed twistor $\sD$-module on $(X,D^+)$.

For a sub-divisor $B=\sum_{i\in I}D_i\subset D$ and
$\bb=(b_i)_{i\in I}\in\R^I$, define
\begin{equation}\label{E:intersection multi V}
\bV^{(\alpha)B}_{\bb}\cN
:=\bigcap_{i\in I}V^{(\alpha)D_i}_{b_i}\cN.
\end{equation}
If $D=B+C$, we will also use
\[
\bV^{(\alpha)B}_{\bb}\bV^{(\alpha)C}_{\bc}\cN
:=\bV^{(\alpha)B}_{\bb}\cN
\cap\bV^{(\alpha)C}_{\bc}\cN.
\]
The filtrations are compatible with the filtration on the ring in the sense that
\begin{equation}\label{E:comp with V on R}
\bV^{(\alpha)D}_{\ba}\cN\cdot
\bV^D_{\bk}\sR_{X(*D^+)}
\subset
\bV^{(\alpha)D}_{\ba+\bk}\cN,
\end{equation}
for any $\ba\in\R^n$ and $\bk\in\Z^n$.

We say $\bb<\ba$ if $b_i<a_i$ for all $i$, and $\bb\prec\ba$ if $b_i\leq a_i$ for all $i$ and $\bb\neq\ba$. We use the following two different lower pieces:
\begin{align}
\bV^{(\alpha)D}_{<\ba}\cN
&:=\bigcup_{\bb<\ba}\bV^{(\alpha)D}_{\bb}\cN
 =\bigcap_{i=1}^nV^{(\alpha)D_i}_{<a_i}\cN,
\label{E:all strictly lower}\\
\bV^{(\alpha)D}_{\prec\ba}\cN
&:=\sum_{i=1}^n
\left(
V^{(\alpha)D_i}_{<a_i}\cN
\cap\bigcap_{j\neq i}V^{(\alpha)D_j}_{a_j}\cN
\right).
\label{E:boundary lower}
\end{align}
We also set
\begin{equation}\label{E:simultaneous graded}
\Gr^{\bV^{(\alpha)D}}_{\ba}\cN
:=\frac{\bV^{(\alpha)D}_{\ba}\cN}
{\bV^{(\alpha)D}_{\prec\ba}\cN}.
\end{equation}
The same notation will be used for every sub-divisor $B\subset D$. If $B$ and $C$ have no common components, we also set
\begin{align*}
\bV^{(\alpha)B}_{\bb}\bV^{(\alpha)C}_{\prec\bc}\cN
:=\sum_{D_j\subset C}
\left(
\bV^{(\alpha)B}_{\bb}\cN
\cap V^{(\alpha)D_j}_{<c_j}\cN
\cap\!\!\bigcap_{\substack{D_k\subset C\\k\neq j}}
V^{(\alpha)D_k}_{c_k}\cN
\right),
\end{align*}
and similarly when the symbol $\prec$ occurs in the $B$-block.

Locally, let $t_i$ define $D_i$, and put
\[
\Theta_i:=t_i\eth_{t_i}.
\]
For each relevant $a\in\R$, after shrinking around the point under consideration, choose a KMS polynomial
\begin{equation}\label{E:KMS polynomial multi}
b^{(\alpha)}_{i,a}(s)
:=\prod_{u\in\cK_i(a,\alpha,P)}
(s+\fe(\lambda,u))^{N_{i,u}}
\in\cO_{\Delta(\alpha,\epsilon)}[s]
\end{equation}
such that right multiplication by $b^{(\alpha)}_{i,a}(\Theta_i)$ annihilates $\Gr^{V^{(\alpha)D_i}}_a\cN$. Here $\Theta_i=t_i\eth_{t_i}$ is the right-module Euler operator corresponding, under transposition, to Mochizuki's left-module operator $-\eth_{t_i}t_i$.

In general, the filtrations $V^{(\alpha)D_i}_\bullet$ need not behave well with each other. The following condition is the relative KMS analogue of multispecialisability without slope, and it is the compatibility condition used in the rest of this paper.

\begin{definition}\label{D:V strict}
Notations as above. We say that $\bV^{(\alpha)D}_\bullet$ is \emph{strictly multispecializable without slope}, if it satisfies the following conditions.
\begin{enumerate}
    \item The $t$-adic relation:
    \begin{equation}\label{E:strict for x_i}
    \bV^{(\alpha)D}_{\ba}\cN
    \xrightarrow{\ \cdot t_i\ }
    \bV^{(\alpha)D}_{\ba-\mathbf1^i}\cN
    \end{equation}
    is an isomorphism for all $\ba\in\R^n$, where
    $\mathbf1^i=(0,\ldots,0,1,0,\ldots,0)$ has the only $1$ at the $i$-th position.

    \item The simultaneous coherence: for every $\ba\in\R^n$,
    $\bV^{(\alpha)D}_{\ba}\cN$ is coherent as an
    $\sR_{(X(*E),D)}$-module. By local discreteness and the $t$-adic relation, it is enough to check this for finitely many indices in a fundamental domain modulo $\Z^n$. Together with these two properties, simultaneous coherence is the goodness condition for the multi-$V$-filtration used below.

    \item The coordinate-wise KMS relation, or the relative condition without slope: for every $i$ and every $\ba\in\R^n$,
    \begin{equation}\label{E:relative without slope}
    \bV^{(\alpha)D}_{\ba}\cN\cdot
    b^{(\alpha)}_{i,a_i}(\Theta_i)
    \subset
    V^{(\alpha)D_i}_{<a_i}\cN
    \cap\bigcap_{j\neq i}V^{(\alpha)D_j}_{a_j}\cN.
    \end{equation}
    In particular, the polynomial in the $i$-th direction is independent of the other indices $a_j$. In the present intersection construction, \eqref{E:relative without slope} follows from the one-variable KMS relation because $\Theta_i$ preserves $V^{(\alpha)D_j}_{a_j}\cN$ for $j\neq i$. We retain it in the definition to display explicitly the coordinate-wise no-slope mechanism.

    \item The simultaneous strictness: for every nonempty sub-divisor $B\subset D$ and every multi-index $\bb$, the module
    \[
    \Gr^{\bV^{(\alpha)B}}_{\bb}\cN
    \]
    is strict over $\cO_{\Delta(\alpha,\epsilon)}$. The cases where $B$ has only one component are already contained in ordinary strict specializability; the additional condition begins with sub-divisors having at least two components.
\end{enumerate}
When the preceding conditions hold on a sufficiently small
neighborhood of $\alpha\in\C_\lambda$, we say that $\cN$ is
$V^{(\alpha)}$-compatible with respect to $D$. In particular,
$V^{(0)}$-compatibility means compatibility on a neighborhood
of $\lambda=0$, and not merely a condition on
$\cN/\lambda\cN$. If the condition holds near every
$\alpha\in\C_\lambda$, we simply say that $\cN$ is
$V$-compatible with respect to $D$.
\end{definition}

\begin{remark}\label{R:strictness extra}
Definition \ref{D:V strict} is a module-theoretic criterion
for the compatibility properties of the single-index
KM-filtrations. Mochizuki defines a multi-indexed
KM-filtration by the intersection lattices
\[
\bV^{(\alpha)D}_{\ba}\cN
=
\bigcap_i V^{(\alpha)D_i}_{a_i}\cN,
\]
together with the lifting epimorphisms and a recursive
condition on the KMS nearby-cycle modules,
\cite[\S 4.1]{MocL2Twistor}. Proposition
\ref{P:relative no slope compatibility} proves from Definition
\ref{D:V strict} the lifting property, distributivity,
stability under taking partial graded modules, and the
independence of the order of the iterated gradings. The
recursive nearby-cycle condition is not used below. For
modules underlying good mixed twistor $\sD$-modules, the full
multi-$V$ property and the strictness of the partial graded
modules follow from
\cite[Proposition 4.7 and \S 4.2.4]{MocL2Twistor}.

Definition \ref{D:V strict} should also be compared with
Kochersperger's strict $\R$-multispecialisability of the Rees
module $R_F\cM$, \cite[Definitions 7.1 and 7.5]{Koc18}.
Strictness of all partial quotients is part of his definition,
and he proves that, in the without-slope setting, this condition
is equivalent to the compatibility of the Hodge filtration and
the single-index KM-filtrations; it is also preserved after
restricting to a subfamily and after taking partial graded
modules, \cite[Propositions 7.7--7.9]{Koc18}. Thus the
simultaneous strictness in Definition \ref{D:V strict} is the
twistor analogue of the strictness condition in the Rees-Hodge
setting, and is additional only when compared with Maisonobe's
unfiltered $\sD$-module definition.

Finally, simultaneous strictness is needed in the relative
B\'ezout argument because the relevant polynomials are, in
general, coprime only over the field of meromorphic germs:
one obtains
\[
A(s)P(s)+C(s)Q(s)=h(\lambda),
\qquad h(\lambda)\neq0,
\]
and strictness allows us to cancel $h(\lambda)$ on the
simultaneous graded quotient.
\end{remark}

We first record the relative B\'ezout argument which will be used repeatedly.

\begin{lemma}[Relative B\'ezout lemma]\label{L:relative Bezout}
Let $A=\cO_{\C_\lambda,\alpha}$ and let $K=\operatorname{Frac}(A)$. Let $\cQ$ be an $A$-module, and let $T$ be an $A$-linear endomorphism of $\cQ$. We write $m\cdot P(T)$ for the corresponding polynomial action. Assume that $P,Q\in A[s]$ are relatively prime in $K[s]$. If $\cQ_1,\cQ_2\subset\cQ$ are $T$-stable submodules and
\[
m\cdot P(T)\in\cQ_1,
\qquad
m\cdot Q(T)\in\cQ_2,
\]
then $m\in\cQ_1+\cQ_2$, provided that $\cQ/(\cQ_1+\cQ_2)$ is strict over $A$.
\end{lemma}

\begin{proof}
Ordinary B\'ezout over the field $K$ gives polynomials $A_1,C_1\in K[s]$ such that
\[
A_1(s)P(s)+C_1(s)Q(s)=1.
\]
After clearing denominators, there exist $A_0,C_0\in A[s]$ and a nonzero $h(\lambda)\in A$ such that
\begin{equation}\label{E:relative Bezout identity}
A_0(s)P(s)+C_0(s)Q(s)=h(\lambda).
\end{equation}
Applying \eqref{E:relative Bezout identity} to $m$, we obtain
\[
h(\lambda)m
=(m\cdot P(T))\cdot A_0(T)
 +(m\cdot Q(T))\cdot C_0(T)
\in\cQ_1+\cQ_2.
\]
Hence the class of $m$ in $\cQ/(\cQ_1+\cQ_2)$ is annihilated by $h(\lambda)$. Since $A$ is a one-dimensional regular local ring and the quotient is strict, multiplication by every nonzero element of $A$ is injective. Therefore the class of $m$ vanishes.
\end{proof}

The next proposition records the main consequences of the relative no-slope condition.

\begin{prop}[Relative no-slope criterion for the multi-$V$-filtration]\label{P:relative no slope compatibility}
Assume that $\cN$ is $V^{(\alpha)}$-compatible with respect to $D$, in the sense of Definition \ref{D:V strict}. Let $B,C\subset D$ be sub-divisors without common components. Then the following statements hold.
\begin{enumerate}
    \item For any multi-indices $\bb$ and $\bc$, we have the distributivity identity
    \begin{equation}\label{E:multi distributivity}
    \bV^{(\alpha)B}_{\bb}\bV^{(\alpha)C}_{\prec\bc}\cN
    +\left(
    \bV^{(\alpha)B}_{\bb}\bV^{(\alpha)C}_{\bc}\cN
    \cap\bV^{(\alpha)B}_{\prec\bb}\cN
    \right)
    =
    \bV^{(\alpha)(B+C)}_{\prec(\bb,\bc)}\cN.
    \end{equation}

    \item For any component $D_i\subset B$, the natural morphism
    \begin{equation}\label{E:lifting multi V}
    \bV^{(\alpha)B}_{\bb}\cN
    \longrightarrow
    \bV^{(\alpha)(B-D_i)}_{\bb_{\hat i}}
    \Gr^{V^{(\alpha)D_i}}_{b_i}\cN
    \end{equation}
    is surjective, where the filtrations on the right hand side are the canonical KM-filtrations with respect to the remaining components.

    \item The filtrations induced on every partial graded module are the canonical KM-filtrations. More precisely, for every sub-divisor $B\subset D$ and every $\bb$, the module
    \[
    \Gr^{\bV^{(\alpha)B}}_{\bb}\cN
    \]
    is $V^{(\alpha)}$-compatible with respect to $(D-B)|_{D_B}$, where $D_B$ denotes the intersection of the components of $B$. In particular,
    \begin{equation}\label{E:iterated equal simultaneous}
    \Gr^{\bV^{(\alpha)C}}_{\bc}
    \Gr^{\bV^{(\alpha)B}}_{\bb}\cN
    \simeq
    \Gr^{\bV^{(\alpha)(B+C)}}_{(\bb,\bc)}\cN.
    \end{equation}
    Hence, the iterated graded modules are independent of the order of the components.
\end{enumerate}
\end{prop}

\begin{proof}
The formal argument is the same as in \cite[Propositions 7.8 and 7.9, Appendix B]{Koc18}. We only explain the point specific to the KMS setting, namely the relative B\'ezout argument needed for the distributivity identity.

We first assume that $C=D_j$ has one component, with index $c$. If $c$ is not a jump, the assertion is immediate. We may therefore assume that $c$ is a jump, and take
\[
m\in
\bV^{(\alpha)B}_{\bb}V^{(\alpha)D_j}_{c}\cN
\cap\bV^{(\alpha)B}_{\prec\bb}\cN.
\]
By exhaustivity and local discreteness, there is a jump $d>c$ such that
\[
m\in\bV^{(\alpha)B}_{\prec\bb}V^{(\alpha)D_j}_{d}\cN.
\]
Let $c=c_0<c_1<\cdots<c_r=d$ be the consecutive jumps, and put
\[
P(s):=\prod_{\nu=1}^r b^{(\alpha)}_{j,c_\nu}(s),
\qquad
Q(s):=b^{(\alpha)}_{j,c}(s).
\]
The coordinate-wise KMS relation gives
\[
m\cdot P(\Theta_j)\in
\bV^{(\alpha)B}_{\prec\bb}V^{(\alpha)D_j}_{c}\cN,
\qquad
m\cdot Q(\Theta_j)\in
\bV^{(\alpha)B}_{\bb}V^{(\alpha)D_j}_{<c}\cN.
\]
The polynomials $P$ and $Q$ are relatively prime over the field of meromorphic germs in $\lambda$. Indeed, an identity
$\fe(\lambda,u)=\fe(\lambda,v)$ between two KMS eigenvalue curves implies $u=v$ by comparing the coefficients of $1$, $\lambda$, and $\lambda^2$, whereas the KMS types occurring in $P$ and $Q$ belong to different filtration levels. Lemma \ref{L:relative Bezout}, together with the simultaneous strictness for $B+D_j$, therefore gives
\[
\bV^{(\alpha)B}_{\bb}V^{(\alpha)D_j}_{c}\cN
\cap\bV^{(\alpha)B}_{\prec\bb}\cN
\subset
\bV^{(\alpha)(B+D_j)}_{\prec(\bb,c)}\cN.
\]
The reverse inclusion is formal. This proves \eqref{E:multi distributivity} when $C$ has one component. The general case follows by the same block induction used in the passage to subfamilies and partial graded modules in \cite[Proposition 7.9]{Koc18}.

Once \eqref{E:multi distributivity} is established, the lifting morphism \eqref{E:lifting multi V} and the order-independence \eqref{E:iterated equal simultaneous} follow from the exact-hypercube argument in \cite[Definition B.4 and Proposition B.6]{Koc18}. The quotient filtrations on a partial graded module satisfy the coherence, $t_i$-adic, and KMS conditions inherited from $\cN$; their graded modules are the corresponding simultaneous graded modules and are strict. Hence uniqueness of the one-variable KM-filtration identifies them with the canonical filtrations. Finally,
\[
\Gr^{\bV^{(\alpha)A}}_{\ba}
\Gr^{\bV^{(\alpha)B}}_{\bb}\cN
\simeq
\Gr^{\bV^{(\alpha)(A+B)}}_{(\ba,\bb)}\cN
\]
for $A\subset D-B$, so stability after taking a partial graded module follows exactly as in \cite[Proposition 7.9]{Koc18}.
\end{proof}

\begin{coro}\label{C:good mixed twistor multi V}
Let $\cM$ be the underlying $\sR_X$-module of a good mixed twistor $\sD$-module on $(X,D^+)$. Then, for every reduced sub-divisor $D\subset D^+$, the localized module $\cM(*D^+)$ is $V$-compatible with respect to $D$.
\end{coro}

\begin{proof}
By \cite[Proposition 4.7]{MocL2Twistor}, the underlying module has a multi-$V$-filtration along $D^+$. In particular, the intersection lattices are coherent and the natural lifting morphisms are surjective. Hence the simultaneous graded modules are identified with the partial iterated graded modules considered in \cite[\S 4.2.4]{MocL2Twistor}. These modules are strict, and successive gradings commute canonically. After localization, the $t_i$-adic relations hold for all indices, while the coordinate-wise KMS relations are the corresponding one-variable relations. Therefore $\cM(*D^+)$ satisfies Definition \ref{D:V strict} with respect to every reduced sub-divisor $D\subset D^+$.
\end{proof}

\begin{remark}\label{R:why strict Bezout}
At a generic value of $\alpha$, the relevant resultants may be invertible, and the argument reduces to the ordinary B\'ezout argument. At an exceptional value, and especially at $\alpha=0$, this need not happen. For example, if $\delta=(1,0)$ and $u\in\R\times\C$, then
\[
\fe(\lambda,u+\delta)=\fe(\lambda,u)-\lambda.
\]
Thus two adjacent KMS factors may differ only by $\lambda$. The relative B\'ezout identity then gives $\lambda[m]=0$ in an obstruction quotient, and simultaneous strictness is exactly what gives $[m]=0$.
\end{remark}

In particular, if $\cM|_{X\setminus D^+}$ underlies a tame harmonic bundle and $\cN=\cM(*D^+)$ is the associated regular-KMS smooth module, then $\cN$ is $V$-compatible; see \cite[Proposition 4.5 and Corollary 4.6]{MocL2Twistor} and Corollary \ref{C:good mixed twistor multi V}.

\begin{lemma}\label{L: V-compatible for smaller B}
Let $\cN$ be a $V^{(\alpha)}$-compatible $\sR_{X(*D^+)}$-module with respect to $D$, and let $B\subset D$ be a reduced simple normal crossing sub-divisor. Then $\cN$ is $V^{(\alpha)}$-compatible with respect to $B$.
\end{lemma}

\begin{proof}
We verify directly that the conditions in Definition \ref{D:V strict} are preserved. Write $D=B+C$ and fix an index $\bc$ for the components of $C$. By the $t$-adic relation, localizing a simultaneous lattice with respect to the local equations of the components of $C$ gives
\begin{equation}\label{E:face by localization}
\bV^{(\alpha)B}_{\bb}\cN
\simeq
\left(
\bV^{(\alpha)B}_{\bb}\bV^{(\alpha)C}_{\bc}\cN
\right)
\otimes_{\sR_{(X(*E),D)}}
\sR_{(X(*(C+E)),B)}.
\end{equation}
Indeed, any local section of the left hand side can be multiplied by a sufficiently large monomial in the local equations of $C$ so that it belongs to the fixed $C$-lattice on the right hand side. Locally, the ring $\sR_{(X(*(C+E)),B)}$ is the Ore localization of $\sR_{(X(*E),D)}$ obtained by inverting the mutually commuting local equations of the components of $C$. Its PBW description shows that this localization is flat as a left module over $\sR_{(X(*E),D)}$. Hence the simultaneous coherence with respect to $B$ follows from that with respect to $D$.

The $t$-adic relations and the coordinate-wise KMS relations for the components of $B$ are inherited directly. Finally, every sub-divisor of $B$ is also a sub-divisor of $D$, so the required simultaneous strictness is already part of the assumptions. Therefore $\cN$ is $V^{(\alpha)}$-compatible with respect to $B$.
\end{proof}

As in the filtered $\sD$-module case in \cite[Lemma 12]{W17a}, we have the following lemma.

\begin{lemma}\label{L: tor van}
Let $H$ be a smooth component of $D$ and denote $D'=D-H$. For any right coherent $\sR_{(X,D)}$-module $\cV$ which is $t$-torsion free, where $t$ is a local holomorphic function defining $H$, we have
\[
\cH^i\left(
\cV\otimes^L_{\sR_{(X,D)}}\sR_{(X,D')}
\right)=0,
\qquad i\neq0.
\]
\end{lemma}

\begin{proof}
Working locally on $X$, we may assume that $X=Y\times\C_t$, $H=Y\times\{0\}$, and $D'=p_Y^*D^Y$ for a simple normal crossing divisor $D^Y\subset Y$. Then
\begin{align*}
\sR_{(X,D)}&=\sR_{(Y,D^Y)}\langle t,t\eth_t\rangle,\\
\sR_{(X,D')}&=\sR_{(Y,D^Y)}\langle t,\eth_t\rangle.
\end{align*}
We also consider $\sR_{(X,D)}\langle\xi\rangle$, which is free as a left $\sR_{(X,D)}$-module with basis $1,\xi,\xi^2,\ldots$, and which has a compatible right $\sR_{(X,D)}$-module structure determined by
\[
[\xi,t]=\lambda,
\qquad
[\xi,t\eth_t]=\lambda\xi.
\]
There is a left $\sR_{(X,D)}$-linear complex
\begin{equation}\label{E: res of log}
0\longrightarrow
\sR_{(X,D)}\langle\xi\rangle
\xrightarrow{\ \cdot(t\xi-t\eth_t)\ }
\sR_{(X,D)}\langle\xi\rangle
\longrightarrow
\sR_{(X,D')}
\longrightarrow0,
\qquad
\xi\longmapsto\eth_t.
\end{equation}
Here the first arrow is right multiplication by $t\xi-t\eth_t$, which is why it is left $\sR_{(X,D)}$-linear.

We give $\xi$, $t\eth_t$, and $\eth_t$ degree $1$. The two maps are strict with respect to the induced filtration, and the associated graded complex is
\[
0\longrightarrow
\sR_{(Y,D^Y)}[t,t\eth_t,\xi]
\xrightarrow{\ \cdot(t\xi-t\eth_t)\ }
\sR_{(Y,D^Y)}[t,t\eth_t,\xi]
\longrightarrow
\sR_{(Y,D^Y)}[t,\eth_t]
\longrightarrow0.
\]
The element $t\xi-t\eth_t$ is monic in the independent variable $t\eth_t$, with coefficient $-1$, and hence multiplication by it is injective. The last term is the quotient by the relation $t\xi=t\eth_t$. Therefore \eqref{E: res of log} is a free resolution of $\sR_{(X,D')}$ as a left $\sR_{(X,D)}$-module.

After tensoring with $\cV$, the only possible negative cohomology is the kernel of
\[
\cV\langle\xi\rangle
\xrightarrow{\ \cdot(t\xi-t\eth_t)\ }
\cV\langle\xi\rangle.
\]
Let $v=\sum_{k=0}^Nv_k\xi^k$ be in this kernel, with $v_N\neq0$. The coefficient of $\xi^{N+1}$ in $v(t\xi-t\eth_t)$ is $v_Nt$. Hence $v_Nt=0$, contradicting the $t$-torsion freeness of $\cV$. Descending induction on the degree of $\xi$ gives $v=0$. This proves the lemma.
\end{proof}

\begin{prop}\label{P:one component log comparison}
Let $\cM$ be a strict coherent $\sR_X$-module. Assume that $\cM(*D)$ is $V^{(\alpha)}$-compatible with respect to $D$, and let $H$ be a smooth component of $D$. Denote $D'=D-H$. For any $\ba'\in\R^{n-1}$, we have
\begin{align}
\bV^{(\alpha)D'}_{\ba'}V^{(\alpha)H}_{0}\cM(*D)
\otimes^L_{\sR_{(X,D)}}\sR_{(X,D')}
&\simeq
\bV^{(\alpha)D'}_{\ba'}\cM^{(\alpha)}[*H](*D'),
\label{E:one comp star}\\
\bV^{(\alpha)D'}_{\ba'}V^{(\alpha)H}_{<0}\cM(*D)
\otimes^L_{\sR_{(X,D)}}\sR_{(X,D')}
&\simeq
\bV^{(\alpha)D'}_{\ba'}\cM^{(\alpha)}[!H](*D').
\label{E:one comp shriek}
\end{align}
Moreover, both $\cM^{(\alpha)}[*H](*D')$ and $\cM^{(\alpha)}[!H](*D')$ are $V^{(\alpha)}$-compatible with respect to $D'$.
\end{prop}

\begin{proof}
Let
\[
A:=\sR_{(X,D)},
\qquad
A':=\sR_{(X,D')}.
\]
Because $V^{(\alpha)H}_{0}\cM(*D)$ and
$V^{(\alpha)H}_{<0}\cM(*D)$ are already localized along $D'$, the natural base-change morphisms identify tensoring from $A$ to $A'$ with tensoring from
$\sR_{(X(*D'),H)}$ to $\sR_{X(*D')}$. Locally, both identifications are computed by the resolution \eqref{E: res of log}. Therefore the one-variable logarithmic comparison gives
\begin{align}
V^{(\alpha)H}_{0}\cM(*D)\otimes^L_{A}A'
&\simeq\cM^{(\alpha)}[*H](*D'),
\label{E:V^H log comparison star}\\
V^{(\alpha)H}_{<0}\cM(*D)\otimes^L_{A}A'
&\simeq\cM^{(\alpha)}[!H](*D').
\label{E:V^H log comparison shriek}
\end{align}
These are the identities in \cite[Lemma 3.1.2, Lemma 3.1.10]{Moc15}, together with Lemma \ref{L: tor van}. We focus on \eqref{E:one comp star}; the proof of \eqref{E:one comp shriek} is the same, with $V^{(\alpha)H}_{0}$ replaced by $V^{(\alpha)H}_{<0}$.

Set
\[
\cL:=V^{(\alpha)H}_{0}\cM(*D),
\qquad
\cL_{\ba'}:=\bV^{(\alpha)D'}_{\ba'}V^{(\alpha)H}_{0}\cM(*D).
\]
We first claim that the natural morphism
\begin{equation}\label{E: Inj of V}
\cL_{\ba'}\otimes_A A'
\longrightarrow
\cL\otimes_A A'
\simeq\cM^{(\alpha)}[*H](*D')
\end{equation}
is injective. It is enough, by Lemma \ref{L: tor van}, to show that $\cL/\cL_{\ba'}$ is $t$-torsion free. Suppose that $m\in\cL$ and $mt\in\cL_{\ba'}$. Since $mt\in V^{(\alpha)H}_{-1}\cM(*D)$, we have
\[
mt\in
\bV^{(\alpha)D'}_{\ba'}V^{(\alpha)H}_{-1}\cM(*D)
=
\cL_{\ba'}\cdot t
\]
by the $t$-adic relation. Hence there exists $m'\in\cL_{\ba'}$ such that $mt=m't$. Since $\cM(*D)$ is $t$-torsion free, $m=m'$, and the claim follows.

For a component $D_i\subset D'$, define a filtration on $\cM^{(\alpha)}[*H](*D')$ by
\begin{equation}\label{E:induced E filtration after star}
F^{D_i}_a
:=\operatorname{Im}\left(
V^{(\alpha)D_i}_aV^{(\alpha)H}_{0}\cM(*D)
\otimes_A A'
\longrightarrow
\cM^{(\alpha)}[*H](*D')
\right).
\end{equation}
The same $t$-torsion argument and Lemma \ref{L: tor van} show that tensoring is exact on the short exact sequences defining the $D_i$-graded modules. More precisely, if the class of $m$ in
$\Gr^{V^{(\alpha)D_i}}_a(V^{(\alpha)H}_{0}\cM(*D))$ is killed by $t$, then
$mt$ belongs to
$V^{(\alpha)D_i}_{<a}V^{(\alpha)H}_{-1}\cM(*D)$; the $t$-adic relation writes it as $m't$ with
$m'\in V^{(\alpha)D_i}_{<a}V^{(\alpha)H}_{0}\cM(*D)$, and cancellation of $t$ gives $m=m'$. Therefore
\begin{equation}\label{E:graded after star}
\Gr^{F^{D_i}}_a\cM^{(\alpha)}[*H](*D')
\simeq
\Gr^{V^{(\alpha)D_i}}_a
\left(V^{(\alpha)H}_{0}\cM(*D)\right)
\otimes_A A'.
\end{equation}
By Proposition \ref{P:relative no slope compatibility}, the module before tensoring on the right hand side is the canonical $V^{(\alpha)H}_{0}$-lattice in
$\Gr^{V^{(\alpha)D_i}}_a\cM(*D)$. Hence \eqref{E:graded after star} is the $[*H]$-extension of the strict module
$\Gr^{V^{(\alpha)D_i}}_a\cM(*D)$ and is strict by the one-variable theory. The coherence, local discreteness, exhaustivity, the $t_i$- and $\eth_{t_i}$-relations, and the KMS relation are inherited from the corresponding properties before tensoring. Consequently, $F^{D_i}_\bullet$ is a KM-filtration along $D_i$. By uniqueness, it is the canonical KM-filtration on $\cM^{(\alpha)}[*H](*D')$.

It remains to check the simultaneous intersections. Let
\[
\cL_i:=V^{(\alpha)D_i}_{a_i'}V^{(\alpha)H}_{0}\cM(*D)
\subset\cL.
\]
For two such submodules, we have an exact sequence
\[
0\longrightarrow\cL_1\cap\cL_2
\longrightarrow\cL_1\oplus\cL_2
\longrightarrow\cL_1+\cL_2
\longrightarrow0.
\]
All three modules are $t$-torsion free, so Lemma \ref{L: tor van} shows that this sequence remains exact after tensoring with $A'$. Together with the injectivity in \eqref{E: Inj of V}, this gives
\[
(\cL_1\cap\cL_2)\otimes_AA'
=(\cL_1\otimes_AA')\cap(\cL_2\otimes_AA')
\]
inside $\cM^{(\alpha)}[*H](*D')$. Induction on the number of components of $D'$ gives
\[
\cL_{\ba'}\otimes_AA'
=
\bigcap_{D_i\subset D'}F^{D_i}_{a_i'}
=
\bV^{(\alpha)D'}_{\ba'}\cM^{(\alpha)}[*H](*D').
\]
This proves \eqref{E:one comp star}.

We finally verify that the resulting module is $V^{(\alpha)}$-compatible with respect to $D'$. Let $G\subset D'$ be a nonempty sub-divisor and let $\mathbf g$ be a multi-index. Put
\[
\cL^G_{\mathbf g}:=
\bV^{(\alpha)G}_{\mathbf g}V^{(\alpha)H}_{0}\cM(*D),
\qquad
\cL^G_{\prec\mathbf g}:=
\bV^{(\alpha)G}_{\prec\mathbf g}V^{(\alpha)H}_{0}\cM(*D).
\]
By Proposition \ref{P:relative no slope compatibility}, we have
\begin{equation}\label{E:G graded H lattice}
\cL^G_{\mathbf g}/\cL^G_{\prec\mathbf g}
\simeq
V^{(\alpha)H}_{0}
\Gr^{\bV^{(\alpha)G}}_{\mathbf g}\cM(*D).
\end{equation}
The module on the right hand side is a submodule of the $H$-localized module
$\Gr^{\bV^{(\alpha)G}}_{\mathbf g}\cM(*D)$, and hence it is $t$-torsion free. Therefore Lemma \ref{L: tor van}, applied to the short exact sequence
\[
0\longrightarrow\cL^G_{\prec\mathbf g}
\longrightarrow\cL^G_{\mathbf g}
\longrightarrow
\cL^G_{\mathbf g}/\cL^G_{\prec\mathbf g}
\longrightarrow0,
\]
shows that taking the simultaneous graded quotient commutes with tensoring by $A'$. We obtain
\begin{align}
\Gr^{\bV^{(\alpha)G}}_{\mathbf g}
\left(\cM^{(\alpha)}[*H](*D')\right)
&\simeq
\Gr^{\bV^{(\alpha)G}}_{\mathbf g}
\left(V^{(\alpha)H}_{0}\cM(*D)\right)
\otimes_AA'\nonumber\\
&\simeq
V^{(\alpha)H}_{0}
\Gr^{\bV^{(\alpha)G}}_{\mathbf g}\cM(*D)
\otimes_AA'.
\label{E:partial graded after star}
\end{align}
The last term is the $[*H]$-extension of the strict simultaneous graded module
$\Gr^{\bV^{(\alpha)G}}_{\mathbf g}\cM(*D)$, and hence is strict. Simultaneous coherence follows from \eqref{E:one comp star}, while the $t_i$-adic and coordinate-wise KMS relations are inherited from those of $\cM(*D)$. Therefore $\cM^{(\alpha)}[*H](*D')$ is $V^{(\alpha)}$-compatible with respect to $D'$.

Replacing $V^{(\alpha)H}_{0}$ by $V^{(\alpha)H}_{<0}$ throughout the argument and using \eqref{E:V^H log comparison shriek} proves \eqref{E:one comp shriek}. The same proof, with $V^{(\alpha)H}_{<0}$ in place of $V^{(\alpha)H}_{0}$ in \eqref{E:G graded H lattice}, identifies every partial simultaneous graded module with the $[!H]$-extension of the corresponding strict graded module of $\cM(*D)$. Hence $\cM^{(\alpha)}[!H](*D')$ is also $V^{(\alpha)}$-compatible with respect to $D'$.
\end{proof}

Applying the previous proposition inductively, we obtain the following logarithmic comparison.

\begin{prop}[Logarithmic Comparison]\label{P:log-comparison}
With the same assumptions as in Proposition \ref{P:one component log comparison}, separate $D$ into two groups of components $D=B+C$. Then, we have
\begin{equation}\label{E:full log comparison}
\bV^{(\alpha)B}_{<\mathbf0}
\bV^{(\alpha)C}_{\mathbf0}\cM(*D)
\otimes^L_{\sR_{(X,D)}}\sR_X
\simeq
\cM^{(\alpha)}[!B][*C]\simeq\cM^{(\alpha)}[*C][!B].
\end{equation}
We denote both of them by $\cM^{(\alpha)}[!B+*C]$.
In particular,
\begin{equation}\label{E:DR log comparison}
\DR_{(X,D)}
\left(
\bV^{(\alpha)B}_{<\mathbf0}
\bV^{(\alpha)C}_{\mathbf0}\cM(*D)
\right)
\simeq
\DR_X\cM^{(\alpha)}[!B+*C].
\end{equation}
\end{prop}

\begin{proof}
Choose an ordering of the components of $D$. By local discreteness, after shrinking around the point under consideration, each term $V^{(\alpha)D_i}_{<0}$ is equal to $V^{(\alpha)D_i}_{-\epsilon_i}$ for a sufficiently small $\epsilon_i>0$. Thus the lattice
\[
\bV^{(\alpha)B}_{<\mathbf0}
\bV^{(\alpha)C}_{\mathbf0}\cM(*D)
\]
may locally be treated as a simultaneous lattice with fixed real indices. We remove the logarithmic structure one component at a time. For a component of $C$, apply \eqref{E:one comp star}; for a component of $B$, apply \eqref{E:one comp shriek}. Proposition \ref{P:one component log comparison} also shows that, after each step, the resulting module is $V^{(\alpha)}$-compatible with respect to the remaining divisor. Hence the remaining multi-indexed KM-filtration is canonical and the induction can continue. Associativity of the derived tensor product gives \eqref{E:full log comparison}, and the final extension is precisely $\cM^{(\alpha)}[!B+*C]$.

For \eqref{E:DR log comparison}, use
\[
\DR_{(X,D)}\cV
=\cV\otimes^L_{\sR_{(X,D)}}\cO_{\cX}
\simeq
\left(
\cV\otimes^L_{\sR_{(X,D)}}\sR_X
\right)
\otimes^L_{\sR_X}\cO_{\cX}
\]
and apply \eqref{E:full log comparison}.
\end{proof}

Recall that the de Rham functor on a right coherent $\sR_X$-module $\cM$ is
\[
\DR_X\cM=\cM\otimes^L_{\sR_X}\cO_{\cX},
\]
which can be explicitly expressed using the Spencer complex $\Sp^\bullet(\cO_{\cX})$, \cite[\S 0.6]{Sab05}, as a resolution of $\cO_{\cX}$ by locally free left $\sR_X$-modules. In the case of right $\sR_{(X,D)}$-modules, we have the logarithmic de Rham functor
\[
\DR_{(X,D)}\cV
=\cV\otimes^L_{\sR_{(X,D)}}\cO_{\cX}.
\]
We can use the logarithmic Spencer complex in \cite[\S 2]{W17a} to get the explicit expression. Please note that the $\Sp$ functor in \emph{loc. cit.} is the same as the de Rham functor here.

Assume that $f:X\to Y$ is a proper map between complex manifolds, and
\[
F=f\times\operatorname{Id}:X\times\C^n\longrightarrow Y\times\C^n.
\]
Let $D^X$ and $D^Y$ be the simple normal crossing divisors on $X\times\C^n$ and $Y\times\C^n$, respectively, defined by the coordinate hyperplanes of $\C^n$. Let
\[
D^X=B^X+C^X,
\qquad
D^Y=B^Y+C^Y
\]
be the naturally corresponding decompositions. Let $\cN$ be a coherent
$\sR_{(X(*C^X),B^X)}$-module on $\cX\times\C^n$. We denote by $\widetilde f:\cX\to\cY$ and $\widetilde F:\cX\times\C^n\to\cY\times\C^n$ the induced morphisms. We use the usual transfer bimodule in the definition
\begin{equation*}
\cH^iF_\dagger\cN
:=R^i\widetilde F_*
\left(
\cN\otimes^L_{\sR_{(X(*C^X),B^X)}}
\sR_{(X(*C^X),B^X)\to(Y(*C^Y),B^Y)}
\right).
\end{equation*}
Let $p_X:X\times\C^n\to X$ and $p_Y:Y\times\C^n\to Y$ denote the first projections; we use the same notation for the induced projections on the corresponding twistor spaces. As in \cite[Remark 1.4.3(2)]{Sab05}, we view the $\sR_{\C^n}$-action as an extra structure. Hence
\begin{equation}\label{E: def of F_dagger}
\cH^iF_\dagger\cN
\simeq
R^i\widetilde F_*
\left(
\cN\otimes^L_{p_X^{-1}\sR_X}
 p_X^{-1}\sR_{X\to Y}
\right).
\end{equation}
This implies that the direct image functor above is compatible with the one in \cite[Theorem 3.1.8]{Sab05}. We can define good $\sR_{(Y(*C^Y),B^Y)}$-modules as in \cite[1.1.c]{Sab05}; by \cite[Theorem 1.4.5]{Sab05}, the proper direct image of a good module is good.

We recall the following consequence of \cite[Theorem 3.1.8 and Theorem 3.3.15]{Sab05} in our setting.

\begin{thm}\label{T: Sabbah's V direct image theorem}
Let $H^X$ be an irreducible coordinate component of $C^X$, and let $H^Y=F(H^X)$. Assume that $\cN$ is a good $\sR_{(X(*C^X),B^X)}$-module and is strictly specializable along $H^X$, with KM-filtration $V^{(\alpha)H^X}_\bullet$. Assume moreover that
\[
\cH^jF_\dagger\Gr^{V^{(\alpha)H^X}}_h\cN
\]
is strict for every $j$ and $h$. Then $\cH^iF_\dagger\cN$ is strictly specializable along $H^Y$, and
\begin{equation}\label{E: V comp under direct image}
\cH^iF_\dagger
\left(V^{(\alpha)H^X}_h\cN\right)
\simeq
V^{(\alpha)H^Y}_h
\left(\cH^iF_\dagger\cN\right)
\end{equation}
for every $h\in\R$. Moreover,
\begin{equation}\label{E: graded V comp under direct image}
\cH^iF_\dagger
\Gr^{V^{(\alpha)H^X}}_h\cN
\simeq
\Gr^{V^{(\alpha)H^Y}}_h
\left(\cH^iF_\dagger\cN\right).
\end{equation}
\end{thm}

We are ready to show the following compatibility of the multi-indexed KM-filtration with the direct image functor. For modules underlying good mixed twistor $\sD$-modules, Mochizuki's Proposition 4.7 gives the intrinsic multi-$V$ property. The additional content below is the lattice-level identity under direct image, which is obtained from the one-variable direct-image theorem.

\begin{prop}\label{P:multi-index V-direct image}
Assume that $f:X\to Y$ and
$F=f\times\operatorname{Id}$ are as above. Fix
$\alpha\in\C_\lambda$, and let $\cM$ be a good strict
$\sR$-module on $X\times\C^n$. Assume that
$\cM(*D^X)$ is $V^{(\alpha)}$-compatible with respect to
$D^X$. Assume moreover that, for every sub-divisor
$B^X\subset D^X$, every multi-index $\bb$, and every $j$,
the module
\begin{equation}\label{E:multi strict direct image assumption}
\cH^jF_\dagger
\Gr^{\bV^{(\alpha)B^X}}_{\bb}\cM(*D^X)
\end{equation}
is strict. For $B^X=0$, the expression in \eqref{E:multi strict direct image assumption} means $\cH^jF_\dagger\cM(*D^X)$. Then $(\cH^iF_\dagger\cM)(*D^Y)$ is $V^{(\alpha)}$-compatible with respect to $D^Y$. Moreover,
\begin{equation}\label{E:multi-V comp under direct image}
\cH^iF_\dagger
\left(
\bV^{(\alpha)D^X}_{\bd}\cM(*D^X)
\right)
\simeq
\bV^{(\alpha)D^Y}_{\bd}
\left((\cH^iF_\dagger\cM)(*D^Y)\right)
\end{equation}
for every $\bd\in\R^n$. Here, the $F_\dagger$ on the left is the direct image functor from logarithmic $\sR$-modules on $(X\times\C^n,D^X)$ to logarithmic $\sR$-modules on $(Y\times\C^n,D^Y)$.
\end{prop}

\begin{proof}
This is the twistor-$\sR$ analogue of \cite[Theorem 8.1]{Koc18}, and we use the proof given there. Notice that Kochersperger considers the same type of morphism $f\times\operatorname{Id}_{\C^n}$. The differences are the following. The one-variable direct-image theorem for the Rees module is replaced by Theorem \ref{T: Sabbah's V direct image theorem}; the compatibility of the source filtrations is supplied by Proposition \ref{P:relative no slope compatibility}; and the strictness obtained there from the mixed Hodge module direct-image theorem is replaced by assumption \eqref{E:multi strict direct image assumption}. The right-module KMS operators and their holomorphic dependence on $\lambda$ are already included in the one-variable theorem, so no further change is needed in the ``cancelling polynomials'' part of Kochersperger's proof.

Set
\[
\cN_X:=\cM(*D^X),
\qquad
\cN_Y^i:=\left(\cH^iF_\dagger\cM\right)(*D^Y).
\]
Since $F$ is the identity on the $\C^n$-factor, localization along the coordinate divisor commutes with $F_\dagger$, and hence $\cH^iF_\dagger\cN_X\simeq\cN_Y^i$. For a sub-divisor $B^X\subset D^X$, with corresponding $B^Y\subset D^Y$, define
\[
U^{B^Y}_{\bb}\cN_Y^i
:=\operatorname{Im}\left(
\cH^iF_\dagger
\bigl(\bV^{(\alpha)B^X}_{\bb}\cN_X\bigr)
\longrightarrow
\cN_Y^i
\right).
\]
Applying the ``Injectivity'', ``Good $V$-multifiltration'', and ``Graded pieces'' arguments of \cite[Theorem 8.1]{Koc18} to the direct-image complex \eqref{E: def of F_dagger}, with the replacements above, gives
\begin{align}
\cH^iF_\dagger
\left(\bV^{(\alpha)B^X}_{\bb}\cN_X\right)
&\simeq U^{B^Y}_{\bb}\cN_Y^i,
\label{E:condensed direct image lattice}\\
\cH^iF_\dagger
\Gr^{\bV^{(\alpha)B^X}}_{\bb}\cN_X
&\simeq
\Gr^{U^{B^Y}}_{\bb}\cN_Y^i.
\label{E:condensed direct image graded}
\end{align}
Moreover, the resulting filtration hypercube is exact. The one-variable theorem identifies the filtration obtained in each coordinate with the canonical KM-filtration on $\cN_Y^i$. Hence
\[
U^{B^Y}_{\bb}\cN_Y^i
=
\bV^{(\alpha)B^Y}_{\bb}\cN_Y^i.
\]
Taking $B^X=D^X$ proves \eqref{E:multi-V comp under direct image}.

Finally, \eqref{E:condensed direct image graded} and assumption \eqref{E:multi strict direct image assumption} give simultaneous strictness on the target. The ``Good $V$-multifiltration'' part of Kochersperger's argument gives simultaneous coherence, while the $t_i$-adic and coordinate-wise KMS relations follow from the corresponding one-variable statements. Therefore $\cN_Y^i$ is $V^{(\alpha)}$-compatible with respect to $D^Y$.
\end{proof}

\begin{remark}\label{R:direct image Mochizuki multi V}
If the direct image is known to underlie a good mixed twistor $\sD$-module, its intrinsic multi-$V$ property follows from \cite[Proposition 4.7]{MocL2Twistor}. The identity \eqref{E:multi-V comp under direct image} is a stronger lattice-level statement and is the part modeled on \cite[Theorem 8.1]{Koc18}.
\end{remark}

Please also refer to \cite[Theorem 3]{W17a} and the remark after its proof. In that case, the map is more general, and we cannot reduce to a map of the form $F$ above. The other results in \cite{W17a} about the direct image and dual functors on logarithmic representations should still work in the setting of $\sR$-modules underlying mixed twistor $\sD$-modules, by essentially the same arguments. Since we do not use them in this paper, we do not copy them here.

Before ending this section, we consider an easy special case. Let $X$ be a complex manifold with a reduced simple normal crossing divisor
$D=D_1+\cdots+D_n$. According to \cite[\S 3.7]{Sab05}, assume that a holonomic $\sR_X$-module $\cM$ is \emph{strictly non-characteristic} with respect to every stratum
\[
D_I:=\bigcap_{i\in I}D_i,
\qquad I\subset\{1,\ldots,n\}.
\]
In particular, successive restrictions in the divisor directions are strictly non-characteristic. Then $\cM(*D)$ is strictly specializable and $V$-compatible, and the multi-indexed KM-filtration $\bV^D_\bullet$ on $\cM(*D)$ is globally defined, i.e. it does not depend on $\alpha$, and it jumps only at $\Z^n$. Writing
\[
(\ba+\mathbf1)D:=\sum_{i=1}^n(a_i+1)D_i,
\]
we have
\begin{equation}\label{E:nonchar localized V}
\bV^D_{\ba}\cM(*D)
=\cM\otimes\cO_{\cX}((\ba+\mathbf1)D),
\qquad \ba\in\Z^n.
\end{equation}

Let $H$ be one component of $D$ and $D':=D-H$. Following \cite[Lemma 3.1.1]{Moc15}, let us first compute
\[
V^H_\bullet\cM(*D')[*H]
:=V^H_\bullet\cM(*D)\cap\cM(*D')[*H].
\]
We claim that
\begin{equation}\label{E: KM fil of nonchar}
V^H_k\cM(*D')[*H]
=
\begin{cases}
\cM(*D')\otimes\cO_{\cX}((k+1)H),& k\leq0,\\[2mm]
\displaystyle
\sum_{j=0}^k
\left(
\cM(*D')\otimes\cO_{\cX}((j+1)H)
\right)\lambda^j,& k\geq1.
\end{cases}
\end{equation}
The case $k\leq0$ follows from \cite[Lemma 3.7.4]{Sab05}. For $k\geq0$, we have
\[
mt^{-k}\cdot\eth_t
=(m\eth_t)t^{-k}+k\lambda mt^{-k-1},
\qquad m\in\cM(*D'),
\]
which gives
\[
V^H_k\cM(*D')[*H]
+V^H_k\cM(*D')[*H]\cdot\eth_t
=V^H_{k+1}\cM(*D')[*H].
\]
Induction gives \eqref{E: KM fil of nonchar}.

Set
\[
\bV^D_{\ba}\cM[*D]
:=\bV^D_{\ba}\cM(*D)\cap\cM[*D].
\]
Because $\cM(*D)$ is $V$-compatible, the single-index filtrations are distributive in the sense of Proposition \ref{P:relative no slope compatibility}. We may therefore intersect the expressions in \eqref{E: KM fil of nonchar} one component at a time. An induction on the number of components of $D$ shows that the resulting filtration jumps only at $\Z^n$ and, for any $\ba\in\Z^n$,
\begin{equation}\label{E: multi KM fil of * ext}
\bV^D_{\ba}\cM[*D]
=
\sum_{\substack{\bk\in\Z^n\\\bk\leq\ba}}
\cM\otimes\cO_{\cX}((\bk+\mathbf1)D)\cdot\lambda^{s_{\bk}},
\end{equation}
where
\[
s_{\bk}:=\sum_{\{i\mid k_i\geq1\}}k_i.
\]
Here $\bk\leq\ba$ means the coordinate-wise order. The sum is locally coherent; terms with sufficiently negative coordinates are contained in the terms with larger coordinates.

\begin{remark}\label{R: non-char}
When $\cM$ is algebraic and holonomic, it is non-characteristic with respect to a general hypersurface by Bertini's theorem; see, for example, \cite[\S 14.3.1.3]{Moc15}. If we further assume that it is strictly specializable, then it is strictly non-characteristic.
\end{remark}

\section{Esnault-Viehweg's covering construction}
  On a smooth variety $X$, assume we have a line bundle $\cL$, with a section $s\in H^0(X, \cL^{\otimes N})$ defining an integral divisor $N(\bd D)$, with $D=D_1+...+D_n$ reduced and normal crossing, $\bd\in \Q^n$, $0 \leq d_i< 1$.
Applying the construction in \cite[\S 3]{EV92}, the section $s$ gives a cyclic covering $g:\hat X\to X,$ that only ramifies along
$\textup{Supp}(\bd D)\subset D$. Let's use $h:\hat X^0\to X^0$ to denote $g$ restricted over $X^0=X\setminus D$, i.e. the \'etale part of $g$.
Fix an $\sR_X$-module $\cM$ that underlies a mixed twistor $\sD$-module. Fix $\alpha\in\C_\lambda$, and assume that
$\cM(*D)$ is $V^{(\alpha)}$-compatible with respect to $D$, and denote $\cM[*D]$, its prolongation along $D$. 
Let $\cM^0=\cM|_{X^0}$, and denote $\hat \cM^0=h^{+}\cM^0$. 

As in \cite[3.1 Notation]{EV92}, we set
$$\cL^{-(i, \bd D)}=(\cL^{(i, \bd D)})^{-1}=(\cL^i(-\lfloor i\bd D\rfloor))^{-1}.$$
We will just use $(i)$ to replace $(i, \bd D)$, if $\bd D$ is clear from the context. We will also use the notation 
$$\left<x\right>=x-\lfloor x\rfloor,$$
for the fractional part of $x$. 

Note that 
\begin{equation}\label{E: cyclic covering decomp}
    h_* \cO_{\hat\sX^0}\simeq \bigoplus_{0\leq i<N} p^* \cL^{-(i)}|_{X^0}
\end{equation}
shall also be viewed as a decomposition of left $\sR$-modules, that underlies mixed twistor $\sD$-modules, corresponds to the cyclic decomposition of $h_*\C_{\hat X^0}$. In particular, $p^* \cL^{-(i)}|_{X^0}$ are not merely line bundles, they carry the rank-one twistor structure induced by the eigenspace decomposition. The eigensummands
$p^*\cL^{-(i)}|_{X^0}$ are left $\sR$-modules underlying
rank-one mixed twistor $\sD$-modules. Consequently, their
tensor products with the right $\sR$-module $\cM^0$ are
naturally right $\sR$-modules. 

Fix an integer $\delta$ such that $0< \delta < N$. Using the notations in \cite[6.1.5]{Moc07a}, locally around a general point of $D_i$, $p^* \cL^{-(\delta)}|_{X^0}$ underlies the tame harmonic bundle $q_1^*L(-\left< \delta d_i \right>,0)$. See also \cite[3.16 Lemma c)]{EV92}. Locally in a simple normal crossing coordinate chart around any point of $D$, the rank-one eigensystem is the tensor product of its one-variable factors. Hence, in the $D_i$ 
direction, after we change it into its corresponding right $\sR$-module, \cite[14.1.2]{Moc07b}, its KM-filtration is given by,
$$V^{D_i}_{a_i} (\omega_{\sX}\otimes p^* \cL^{-(\delta)}(*D_i))=\omega_{\sX}\otimes p^*\cL^{-(\delta)}(kD_i), $$
if $-\left< \delta d_i \right>+k-1\leq a_i< -\left< \delta d_i \right>+k, k\in \Z.$
Recall that $\omega_{\sX}:=\lambda^{-\dim X}\cdot p^*\omega_X,$ e.g. \cite[Example 14.4]{Moc07b}. 

Denote $j:X^0\to X,$ the natural embedding, and denote 
$$ \cM_{\delta}=j_+ (\cM^0\otimes p^* \cL^{-(\delta)}|_{X^0}),$$
which is an $\sR$-module that underlies a mixed twistor $\sD$-module on $X$, satisfying  $  \cM_{\delta}=  \cM_{\delta}[*D].$

By using the functoriality for the tensor product in \cite[7.2.6]{Moc07a}, we have
\begin{align*}
    V^{(\alpha)D_i}_{a_i}  \cM _{\delta}(*D)= & V^{(\alpha)D_i}_{\left< \delta d_i \right>+a_i}\cM (*D) \otimes V^{(\alpha)D_i}_{-\left< \delta d_i \right>-1} (\omega_{\sX}\otimes p^*\cL^{-(\delta)}(*D)) \otimes  \omega^{-1}_{\sX} \\
    \simeq & V^{(\alpha)D_i}_{\left< \delta d_i \right>+a_i}\cM (*D) \otimes p^* \cL^{-(\delta)}(*(D-D_i)).
\end{align*}
The shifting of degree $1$ in the KM-filtration is due to the convention of the shifting from the parabolic structure to the KM-filtration as in \cite[15.1.2.]{Moc07b}. Twisting $\omega^{-1}_{\sX}$ at the end is to make the tensor product works between right and left modules $\sR$-modules.  
By taking intersections of those KM-filtrations respect to all components of $D$, we get 
\begin{align}\label{E: KM fil of twist line bundle}
    \bV^{(\alpha)D}_{\ba}  \cM _{\delta}(*D)= & \bV^{(\alpha)D}_{\left<\delta \bd \right>+\ba}\cM (*D) \otimes \bV^{(\alpha)D}_{-\left<\delta \bd \right>-\1} (\omega_{\sX}\otimes p^*\cL^{-(\delta)}(*D)) \otimes  \omega^{-1}_{\sX} \\
    \simeq & \bV^{(\alpha)D}_{\left<\delta \bd \right>+\ba}\cM (*D) \otimes p^* \cL^{-(\delta)}.\nonumber
\end{align}
In particular, we have the following
\begin{lemma}\label{L:cyclic computation with V}
In the above setting, $\cM_{\delta}=\cM_{\delta}[*D]$, and $\cM_{\delta}(*D)$ is $V^{(\alpha)}$-compatible with respect to $D$. Then, we have:
$$\bV^{(\alpha)D}_{\delta \bd }\cM (*D)\otimes p^*\cL^{-\delta}\simeq \bV^{(\alpha)D}_{\left<\delta \bd \right>}\cM (*D)\otimes p^*\cL^{-(\delta)}\simeq \bV^{(\alpha)D}_{\0}  \cM _{\delta}(*D),$$ 
and
$$\bV^{(\alpha)D}_{<\delta \bd }\cM (*D)\otimes p^*\cL^{-\delta}\simeq \bV^{(\alpha)D}_{<\left<\delta \bd \right>}\cM (*D)\otimes p^*\cL^{-(\delta)}\simeq \bV^{(\alpha)D}_{<\0}  \cM _{\delta}(*D),$$ 
as $\sR_{(X,D)}$-modules. We equip $p^*\cL^{-\delta}$ with the logarithmic twistor
structure obtained from the eigensystem
$p^*\cL^{-(\delta)}|_{X^0}$ by the corresponding integral
shift of its logarithmic lattice.
\end{lemma}
\begin{proof}
We just show that $\cM_\delta(*D)$ is $V^{(\alpha)}$-compatible with
respect to $D$ and the rest of the statements are clear by the previous computation. 

Set $\mathbf r:=\left<\delta\bd\right>.$
More generally, for every reduced sub-divisor
$A\subset D$, denoting by $\mathbf r_A$ the restriction of
$\mathbf r$ to the components of $A$, equation
\eqref{E: KM fil of twist line bundle} gives
\begin{equation}\label{E:partial KM twist}
\bV^{(\alpha)A}_{\ba}\cM_\delta(*D)
\simeq
\bV^{(\alpha)A}_{\ba+\mathbf r_A}\cM(*D)
\otimes p^*\cL^{-(\delta)}.
\end{equation}
Here both sides are already localized along the components of
$D-A$.

Tensoring with the rank-one logarithmic twistor structure
$p^*\cL^{-(\delta)}$ is an exact autoequivalence of the
category of logarithmic $\sR$-modules. Hence the simultaneous
coherence and the $t_i$-adic relations for
$\cM_\delta(*D)$ follow from those for $\cM(*D)$. The
coordinate-wise KMS relations are also preserved by the
one-variable tensor-product rule: in the $i$-th direction the
KMS type is translated by the KMS type of the rank-one factor,
and this translation is independent of the indices in the
other directions.

Moreover, taking the simultaneous graded module in
\eqref{E:partial KM twist}, we obtain
\[
\Gr^{\bV^{(\alpha)A}}_{\ba}\cM_\delta(*D)
\simeq
\Gr^{\bV^{(\alpha)A}}_{\ba+\mathbf r_A}\cM(*D)
\otimes p^*\cL^{-(\delta)}.
\]
The right hand side is strict, because the corresponding
simultaneous graded module of $\cM(*D)$ is strict and the
rank-one factor is locally free over
$\cO_{\Delta(\alpha,\epsilon)}$. Thus
$\cM_\delta(*D)$ satisfies all the conditions of Definition
\ref{D:V strict}.
\end{proof}
More generally, the same argument allows closed indices in
some components and open indices in the remaining components.
In particular, if
\[
\bd D=\bb B+\bc C,
\]
then
\begin{equation}\label{E:mixed cyclic computation}
\bV^{(\alpha)B}_{\delta\bb}
\bV^{(\alpha)C}_{<\delta\bc}
\cM(*D)\otimes p^*\cL^{-\delta}
\simeq
\bV^{(\alpha)B}_{\mathbf0}
\bV^{(\alpha)C}_{<\mathbf0}
\cM_\delta(*D).
\end{equation}

We can also generalize an intermediate result in Saito's proof of Saito vanishing \cite[(2.33.3)]{Sa90}, which will be used in the proof of Theorem \ref{T:twistor vanishing}.
\begin{lemma}\label{L:cyclic computation saito's version}
In the above setting, if we further assume that the $\sR_X$-module $\cM$ is strictly non-characteristic with respect to every stratum of $D$, then we have that, for any integer $\delta$, $0< \delta < N$, satisfying all components of $\delta \bd$ are not integers, we have $\cM_{\delta}=  \cM_{\delta}[*D]=  \cM_{\delta}[!D]$. Furthermore, we have 
$$\Xi_{\Dol}\cM_{\delta}\simeq \Xi_{\Dol}\cM[*D]\otimes \cL^{-(\delta)} \simeq \Xi_{\Dol}\cM[!D]\otimes \cL^{-(\delta)}(D),$$
as $\sA_X$-modules, where $\cL^{-(\delta)}$ and $\cL^{-(\delta)}(D)$ carry trivial higgs structure, i.e. differential operators act trivially on them. 
\end{lemma}
\begin{proof}
Since $\cM$ is strictly non-characteristic with respect to
every stratum of $D$, its multi-indexed KM-filtration jumps
only at $\Z^n$. 
Taking $\ba=\mathbf0$ in
\eqref{E: KM fil of twist line bundle}, we obtain
\begin{equation}\label{E:V0 equal Vless0 twisted}
\bV^D_{\mathbf0}\cM_\delta(*D)
=
\bV^D_{<\mathbf0}\cM_\delta(*D).
\end{equation}

Apply Proposition \ref{P:log-comparison} first with
$(B,C)=(0,D)$ and then with $(B,C)=(D,0)$. From
\eqref{E:V0 equal Vless0 twisted}, we obtain
\[
\cM_\delta[*D]
\simeq
\bV^D_{\mathbf0}\cM_\delta(*D)
\otimes^L_{\sR_{(X,D)}}\sR_X
\simeq
\bV^D_{<\mathbf0}\cM_\delta(*D)
\otimes^L_{\sR_{(X,D)}}\sR_X
\simeq
\cM_\delta[!D].
\]
By the definition of $\cM_\delta$, we already have
$\cM_\delta=\cM_\delta[*D]$. Hence
\[
\cM_\delta
=
\cM_\delta[*D]
=
\cM_\delta[!D].
\]

We now prove the second assertion. By the non-characteristic
formula \eqref{E:nonchar localized V}, we have
\[
\bV^D_{\mathbf0}\cM(*D)
=
\bV^D_{<\mathbf0}\cM(*D)
\otimes p^*\cO_X(D).
\]
Consequently,
\begin{align}\label{E:three logarithmic lattices}
\bV^D_{\mathbf0}\cM_\delta(*D)
&\simeq
\bV^D_{\mathbf0}\cM(*D)
\otimes p^*\cL^{-(\delta)}
\nonumber\\
&\simeq
\bV^D_{<\mathbf0}\cM(*D)
\otimes p^*\cL^{-(\delta)}(D).
\end{align}

Consider the functor
\[
\Phi_D(\cV)
:=
\Xi_{\Dol}\left(
\cV\otimes^L_{\sR_{(X,D)}}\sR_X
\right).
\]
By Proposition \ref{P:log-comparison},
\[
\Phi_D\left(\bV^D_{\mathbf0}\cM_\delta(*D)\right)
\simeq
\Xi_{\Dol}\cM_\delta,
\]
\[
\Phi_D\left(\bV^D_{\mathbf0}\cM(*D)\right)
\simeq
\Xi_{\Dol}\cM[*D],
\]
and
\[
\Phi_D\left(\bV^D_{<\mathbf0}\cM(*D)\right)
\simeq
\Xi_{\Dol}\cM[!D].
\]
Moreover, the logarithmic Spencer resolution shows that
$\Phi_D$ commutes with tensoring by the line bundles
$\cL^{-(\delta)}$ and $\cL^{-(\delta)}(D)$, which carry the
trivial Higgs structure after restricting to $\lambda=0$.
Applying $\Phi_D$ to \eqref{E:three logarithmic lattices}, we
obtain
\[
\Xi_{\Dol}\cM_\delta
\simeq
\Xi_{\Dol}\cM[*D]\otimes\cL^{-(\delta)}
\simeq
\Xi_{\Dol}\cM[!D]\otimes\cL^{-(\delta)}(D).
\]
This proves the lemma.
\end{proof}

\section{Proofs of main theorems}

We first recall the non-characteristic hyperplane resolution used below. Let $X$ be a smooth algebraic variety of dimension $n$, let $\cA$ be a base point free line bundle, and let $\cK$ be an algebraic $\sR_X$-module underlying a mixed twistor $\sD$-module. Choose general smooth members successively to obtain
\[
X=Y_0\supset Y_1\supset\cdots\supset Y_n,
\qquad
\cO_{Y_{j-1}}(Y_j)\simeq\cA|_{Y_{j-1}},
\]
where the inclusions are strictly non-characteristic for the successive restrictions of $\cK$. Once the construction terminates, all later $Y_j$ are understood to be empty. Put $Y_{n+1}=\emptyset$, and denote by $i_j:Y_j\hookrightarrow X$ the natural embedding. By \cite[Lemmas 3.1.23 and 3.1.24]{Moc15}, we have exact sequences
\begin{align}
0\longrightarrow\cK\longrightarrow\cK[*Y_1]
&\longrightarrow i_{1,+}\bigl(\cK|_{Y_1}[*Y_2]\bigr)
\longrightarrow\cdots\longrightarrow i_{n,+}\cK|_{Y_n}
\longrightarrow0,
\label{E:star hyperplane resolution}\\
0\longrightarrow i_{n,+}\cK|_{Y_n}
&\longrightarrow\cdots\longrightarrow
 i_{1,+}\bigl(\cK|_{Y_1}[!Y_2]\bigr)
\longrightarrow\cK[!Y_1]\longrightarrow\cK
\longrightarrow0.
\label{E:shriek hyperplane resolution}
\end{align}

Recall that, for $\alpha\in\C^*$, we use
\[
\Xi_{\DR, \alpha}\cM
:=
\cM/(\lambda-\alpha)\cM,
\]
viewed as a $\sD_X$-module after the natural rescaling of the
differential operators.
We have the compatibility of the prolongation functors $[*E], [!E]$ in the category of mixed twistor $\sD$-modules and in the category of holonomic $\sD$-modules, \cite[Lemma 3.3.15]{Moc15}:
\begin{lemma}\label{L: Compatibility of prolongations}
    Let $E$ be a reduced divisor on $X$. Let $\alpha\in\C^*$ be generic with respect to $E,$ in the sense of \cite[2.1.2.4]{Moc15}.
    Then, for any $\sR_X$-module $\cM$ that underlies a mixed twistor $\sD$-module, we have 
    \begin{align*}
        \Xi_{\DR, \alpha}(\cM[*E])&\simeq (\Xi_{\DR, \alpha}\cM)[*E],\\
        \Xi_{\DR, \alpha}(\cM[!E])&\simeq (\Xi_{\DR, \alpha}\cM)[!E].
    \end{align*}
\end{lemma}

\begin{proof}[Proof of Theorem \ref{T:twistor vanishing}]
We prove the second vanishing. For the first vanishing, we apply the dual hyperplane
resolution \eqref{E:shriek hyperplane resolution} to
$\cM[*E]$ and use the eigensummand $\delta=N-1$ in the cyclic
covering construction. Lemma
\ref{L:cyclic computation saito's version} gives the same
Artin-vanishing argument with ordinary, rather than compactly
supported, cohomology.

Choose $N>1$ and an integer $e>0$ such that
$\cA:=\cL^{\otimes N}$
is base point free and $\cA(eE)$ is ample. Such a choice is
possible: if the integer $d$ in the statement is positive, we
may take $e=d$; if $d=0$, then $\cL$ is ample, and
$\cL^{\otimes N}(E)$ is ample for $N\gg0$.
Apply \eqref{E:star hyperplane resolution} to $\cM[!E]$. After applying $\Xi_{\Dol}$ and tensoring by $\cL^{-1}$, we obtain an exact sequence
\begin{equation}\label{E:Higgs affine resolv}
0\longrightarrow
\Xi_{\Dol}\cM[!E]\otimes\cL^{-1}
\longrightarrow\cE^0\longrightarrow\cdots\longrightarrow\cE^n
\longrightarrow0,
\end{equation}
where, writing $E_j:=E|_{Y_j}$,
\[
\cE^j=
\Xi_{\Dol}i_{j,+}
(\cM[!E]|_{Y_j}[*Y_{j+1}])
\otimes \cL^{-1}.
\]

If $Y_{j+1}\neq\emptyset$, apply Lemma \ref{L:cyclic computation saito's version} on $Y_j$ with the smooth divisor $Y_{j+1}$, coefficient $1/N$, and $\delta=1$. If $Y_{j+1}=\emptyset$, the same conclusion follows from the corresponding rank-one $N$-torsion twist. In either case, we obtain an $\sR_{Y_j}$-module $\cN_j$ underlying a mixed twistor $\sD$-module such that
\begin{equation}\label{E:Ej cyclic identification}
\cE^j\simeq\Xi_{\Dol}i_{j,+}\cN_j[!E_j],
\qquad
\cN_j[!E_j]
=\cN_j[!E_j][!Y_{j+1}]
=\cN_j[!E_j][*Y_{j+1}].
\end{equation}
Since
\[
\cO_{Y_j}(Y_{j+1}+eE_j)
\simeq\cA|_{Y_j}(eE_j)
\]
is ample, the open set
\[
Z_j:=Y_j\setminus\operatorname{Supp}(E_j+Y_{j+1})
\]
is affine. 

Choose $\alpha\in\C^*$ generic with respect to all the
prolongations along $E_j$ and $Y_{j+1}$ occurring above. By
Lemma \ref{L: Compatibility of prolongations}, we have
\[
\Xi_{\DR, \alpha}
\bigl(\cN_j[!E_j]\bigr)
\simeq
\bigl(\Xi_{\DR, \alpha}\cN_j\bigr)[!E_j],
\]
and the analogous identities after the prolongations along
$Y_{j+1}$. Hence Artin vanishing gives
\[
\begin{aligned}
\mathbb H^k
\left(
Y_j,
\DR_{Y_j}
\Xi_{\DR, \alpha}\cN_j[!E_j]
\right)
&=
\mathbb H_c^k
\left(
Z_j,
\DR_{Z_j}
\Xi_{\DR, \alpha}\cN_j|_{Z_j}
\right)\\
&=0,
\qquad k<0.
\end{aligned}
\]
The cohomology modules of the proper direct image of
$\cN_j[!E_j]$ to $\C_\lambda$ are strict, and hence locally
free over $\C_\lambda$. Since their fiber at $\lambda=\alpha$
vanishes in degree $k<0$, their fiber at $\lambda=0$ also
vanishes. Therefore
\[
\mathbb H^k(X,\DR_X\cE^j)=0,
\qquad k<0.
\]
The exact resolution \eqref{E:Higgs affine resolv} gives a
spectral sequence
\[
E_1^{p,q}
=
\mathbb H^q(X,\DR_X\cE^p)
\Longrightarrow
\mathbb H^{p+q}
\left(
X,
\DR_X
\left(
\Xi_{\Dol}\cM[!E]\otimes\cL^{-1}
\right)
\right).
\]
Since $p\geq0$ and $E_1^{p,q}=0$ for $q<0$, the abutment
vanishes in negative total degree.
\end{proof}

\begin{proof}[Proof of Theorem
\ref{T:relative twistor vanishing}]
We prove the second vanishing; the first one follows from the
dual construction.

Choose an ample line bundle $\cH$ on $S$. After replacing
$\cL$ by
\[
\cL\otimes f^*\cH^{m}
\]
for $m\gg0$, we may assume that $\cL$ is semi-ample on $X$.
Indeed, this replacement does not affect the desired vanishing,
by the projection formula.

Choose $N>1$ and an integer $e>0$ such that
\[
\cA:=\cL^{\otimes N}
\]
is base point free and $\cA(eE)$ is $f$-ample, as in the proof
of Theorem \ref{T:twistor vanishing}. Applying
\eqref{E:star hyperplane resolution} to $\cM[!E]$, and then
applying $\Xi_{\Dol}$ and tensoring by $\cL^{-1}$, we obtain an
exact sequence
\begin{equation}\label{E:res in relative vanishing}
0\longrightarrow
\Xi_{\Dol}\cM[!E]\otimes\cL^{-1}
\longrightarrow
\cE^0\longrightarrow\cdots\longrightarrow\cE^n
\longrightarrow0,
\end{equation}
where the terms $\cE^j$ are as in
\eqref{E:Ej cyclic identification}.

Let
\[
f_j:=f|_{Y_j},
\qquad
g_j:Z_j\longrightarrow S
\]
be the induced morphisms. Since
\[
\cO_{Y_j}(Y_{j+1}+eE_j)
\]
is $f_j$-ample, the morphism $g_j$ is affine.

Choose $\alpha\in\C^*$ generic with respect to all the
prolongations along $E_j$ and $Y_{j+1}$ occurring above.
By Lemma \ref{L: Compatibility of prolongations} and relative
Artin vanishing, we have
\[
{}^p\!\cH^q
g_{j,!}
\left(
\DR_{Z_j}
\Xi_{\DR,\alpha}\cN_j|_{Z_j}
\right)=0,
\qquad q<0.
\]
Using
\eqref{E:Ej cyclic identification} and the
Riemann--Hilbert correspondence, we obtain
\[
\Xi_{\DR,\alpha}
\cH^q(f_j)_\dagger
\bigl(\cN_j[!E_j]\bigr)
=0,
\qquad q<0.
\]
The module
\[
\cH^q(f_j)_\dagger
\bigl(\cN_j[!E_j]\bigr)
\]
underlies a mixed twistor $\sD$-module. Since
$\Xi_{\DR,\alpha}$ is faithful on this category,
\cite[Remark 7.2.9]{Moc15}, it follows that
\[
\cH^q(f_j)_\dagger
\bigl(\cN_j[!E_j]\bigr)
=0,
\qquad q<0.
\]

Specializing at $\lambda=0$ and using the compatibility of
$\Xi_{\Dol}$ with projective direct image, we obtain
\[
R^qf_*
\left(
\cE^j
\otimes_{\sA_X}^L
f^*\sA_S
\right)=0,
\qquad q<0.
\]
The exact sequence
\eqref{E:res in relative vanishing} induces a spectral sequence
\[
E_1^{j,q}
=
R^qf_*
\left(
\cE^j
\otimes_{\sA_X}^L
f^*\sA_S
\right)
\Longrightarrow
R^{j+q}f_*
\left(
\left(
\Xi_{\Dol}\cM[!E]
\otimes_{\sA_X}^L
f^*\sA_S
\right)
\otimes\cL^{-1}
\right).
\]
Since $j\geq0$ and $E_1^{j,q}=0$ for $q<0$, the abutment
vanishes in every negative degree. This proves the second
vanishing.

For the first vanishing, we use
\eqref{E:shriek hyperplane resolution}, the eigensummand
$\delta=N-1$, and relative Artin vanishing for ordinary direct
image. The same argument gives the required vanishing in every
positive degree.
\end{proof}

\begin{proof}[Proof of Theorem \ref{T:KV vanishing}]
By the two remarks following the statement, we may assume that
$A$ is ample and that $a,\bb,\bc$ are rational. Using Remark
\ref{R:shift V fil}, we absorb the integral parts of $\bb$ and
$\bc$ into $L$ and assume
\[
\mathbf0\leq\bb,\bc<\mathbf1.
\]
We work near $\lambda=0$ and omit the superscript $(0)$. Put
\[
\bd D:=\bb B+\bc C.
\]

\medskip
\noindent
\emph{Reduction to $\Q$-linear equivalence.}
Set
\[
Q:=aA+\bb B+\bc C.
\]
Choose $m>0$ sufficiently divisible so that $mQ$ is integral
and
\[
\cO_X(mL)\otimes\cP\simeq\cO_X(mQ)
\]
for some $\cP\in\Pic^0(X)$. Since multiplication by $m$ on
$\Pic^0(X)$ is surjective, there exists
$\cP_0\in\Pic^0(X)$ such that
\[
\cP_0^{\otimes m}\simeq\cP.
\]
Choose an integral divisor $\widetilde L$ satisfying
\[
\cO_X(\widetilde L)
\simeq
\cO_X(L)\otimes\cP_0.
\]
Then
\[
m\widetilde L\sim mQ,
\qquad\text{and hence}\qquad
\widetilde L\sim_{\Q}Q.
\]

For the second vanishing, we replace
\[
(\cM,L)
\quad\text{by}\quad
(\cM\otimes p^*\cP_0,\widetilde L),
\]
whereas for the first vanishing we replace it by
\[
(\cM\otimes p^*\cP_0^{-1},\widetilde L).
\]
In either case the complex to be considered is unchanged.
Moreover, these rank-one unitary twists are nonsingular along
$D$, and therefore preserve $V^{(0)}$-compatibility and do not
shift the KM-indices. Since the two assertions are independent,
we make the appropriate replacement for each of them and
continue to denote the resulting objects by $\cM$ and $L$.
Thus we may assume
\begin{equation}\label{E:assume Q-linear}
L\sim_{\Q}aA+\bb B+\bc C.
\end{equation}

\medskip
\noindent
\emph{Bloch--Gieseker reduction.}
Write
\[
a=\frac{p}{q},
\qquad p,q\in\Z_{>0}.
\]
By the Bloch--Gieseker covering theorem, there exist a finite
flat morphism
\[
\rho:X'\longrightarrow X
\]
from a smooth projective variety and an ample integral divisor
$A_0$ on $X'$ such that
\[
qA_0\sim p\rho^*A.
\]
We choose the covering in general position with respect to the
normal crossing stratification of $D$ and the relevant
characteristic varieties. Thus, on every normal crossing chart
of $D$, the morphism $\rho$ is locally of the form
\[
h\times\operatorname{Id}_{\Delta^r}
\]
in the normal directions to the local branches of $D$.

Put
\[
D':=(\rho^{-1}D)_{\red}=B'+C',
\qquad
\cM'(*D'):=\rho^\dagger\cM(*D).
\]
Then $\cM'(*D')$ is $V^{(0)}$-compatible, and its
multi-indexed KM-lattices are the pullbacks of the corresponding
lattices of $\cM(*D)$. The strictness hypothesis in Proposition
\ref{P:multi-index V-direct image} holds because the partial
graded modules underlie mixed twistor $\sD$-modules and finite
direct image preserves strictness.

By the projection formula,
\[
\rho_\dagger\cM'(*D')
\simeq
\cM(*D)\otimes
\rho_\dagger\cO_{\cX'}(*D').
\]
The normalized trace morphism splits the natural inclusion
\[
\cO_{\cX}(*D)
\longrightarrow
\rho_\dagger\cO_{\cX'}(*D').
\]
Hence $\cM(*D)$ is a direct summand of
$\rho_\dagger\cM'(*D')$. Proposition
\ref{P:multi-index V-direct image} gives the corresponding
splitting on all the multi-indexed KM-lattices. After applying
$\Xi_{\Dol}$, logarithmic de Rham, and the projection formula,
the complexes on $X$ are therefore direct summands of the
corresponding direct images from $X'$. Consequently, it is
enough to prove the assertions on $X'$.

Replacing
\[
(X,D,B,C,\cM,L)
\quad\text{by}\quad
(X',D',B',C',\cM',\rho^*L)
\]
and omitting the primes, we may assume that there is an ample
integral divisor $A_0$ satisfying
\[
A_0\sim_{\Q}aA.
\]
Set
\[
L_1:=L-A_0.
\]
Then $L_1$ is integral and
\[
L_1\sim_{\Q}\bb B+\bc C.
\]
Choose $N>1$ sufficiently divisible so that
\[
NL_1\sim N(\bb B+\bc C).
\]
Put
\[
\cA:=\cO_X(A_0),
\qquad
\cL:=\cO_X(L_1),
\]
and let $\cM_\delta$ be the eigensystem associated with the
cyclic-covering data
\[
\cL^{\otimes N}\simeq\cO_X(N\bd D)
\]
as in \S4.

\medskip
\noindent
\emph{The second vanishing.}
By \eqref{E:mixed cyclic computation}, we have
\[
\bV^B_{\bb}\bV^C_{<\bc}\cM(*D)
\otimes p^*\cL^{-1}
\simeq
\bV^B_{\mathbf0}\bV^C_{<\mathbf0}\cM_1(*D).
\]
Since
\[
\cO_X(-L)=\cL^{-1}\otimes\cA^{-1},
\]
logarithmic comparison gives
\begin{equation}\label{E:V twist L id with log comp}
\DR_{(X,D)}
\left(\Xi_{\Dol}\bV^B_{\bb}\bV^C_{<\bc}\cM(*D)\right)\otimes\cO_X(-L)\simeq\DR_X\left(\Xi_{\Dol}\cM_1[!C+*B]\right)\otimes\cA^{-1}.
\end{equation}

The right hand side has vanishing hypercohomology in negative
degrees by Theorem \ref{T:twistor vanishing}, applied to
$\cM_1[*B]$ with boundary $C$ and the ample line bundle
$\cA$. This proves the second assertion.

\medskip
\noindent
\emph{The first vanishing.}
Let
\[
D_+:=\operatorname{Supp}(\bd D),
\]
and let $\mathbf1_+\in\{0,1\}^n$ be the coefficient vector of
$D_+$ in $D$. We denote by $\mathbf1_{+,B}$ and
$\mathbf1_{+,C}$ its restrictions to the components of $B$
and $C$, respectively.

Since $N\bd\in\Z^n$ and $\mathbf0\leq\bd<\mathbf1$, we have
\[
\left\lfloor(N-1)\bd D\right\rfloor
=
N\bd D-D_+,
\qquad
\left<(N-1)\bd\right>
=
\mathbf1_+-\bd.
\]
By the definition of the Esnault--Viehweg eigensheaf,
\begin{align*}
\cL^{-(N-1,\bd D)}
&=
\left(
\cL^{\otimes(N-1)}
\otimes
\cO_X\left(
-\left\lfloor(N-1)\bd D\right\rfloor
\right)
\right)^{-1}\\
&=
\cL^{\otimes(1-N)}
\otimes
\cO_X\left(
\left\lfloor(N-1)\bd D\right\rfloor
\right)\\
&\simeq
\cL(-D_+),
\end{align*}
where we use
\[
\cL^{\otimes N}\simeq\cO_X(N\bd D).
\]

The mixed cyclic-filtration formula therefore gives
\[
\bV^B_{<\mathbf0}\bV^C_{\mathbf0}
\cM_{N-1}(*D)
\simeq
\bV^B_{<\mathbf1_{+,B}-\bb}
\bV^C_{\mathbf1_{+,C}-\bc}
\cM(*D)
\otimes p^*\cL(-D_+).
\]
On the other hand, the integral $t_i$-adic shifts give
\[
\bV^B_{<\mathbf1_{+,B}-\bb}
\bV^C_{\mathbf1_{+,C}-\bc}
\cM(*D)
\simeq
\bV^B_{<-\bb}
\bV^C_{-\bc}
\cM(*D)
\otimes p^*\cO_X(D_+).
\]
The two divisor twists cancel, and hence
\begin{equation}\label{E:first cyclic identity}
\bV^B_{<-\bb}\bV^C_{-\bc}\cM(*D)
\otimes p^*\cL
\simeq
\bV^B_{<\mathbf0}\bV^C_{\mathbf0}
\cM_{N-1}(*D).
\end{equation}

Since
\[
\cO_X(L)=\cL\otimes\cA,
\]
logarithmic comparison and \eqref{E:first cyclic identity}
give
\begin{equation}\label{E:V twist L id with log comp'}
    \DR_{(X,D)}\left(\Xi_{\Dol}\bV^B_{<-\bb}\bV^C_{-\bc}\cM(*D)\right)\otimes\cO_X(L)\simeq\DR_X\left(\Xi_{\Dol}\cM_{N-1}[!B+*C]
\right)\otimes\cA.
\end{equation}
The right hand side has vanishing hypercohomology in positive
degrees by Theorem \ref{T:twistor vanishing}, applied to
$\cM_{N-1}[!B]$ with boundary $C$ and the ample line bundle
$\cA$. This proves the first assertion.
\end{proof}

\begin{proof}[Proof of Theorem \ref{T:relative KV vanishing}]
The same argument in the Proof of Theorem \ref{T:KV vanishing} works relative to $S$. By the remarks following the statement of Theorem \ref{T:relative KV vanishing}, we first reduce to the case where $A$ is $f$-ample and all coefficients are rational. Apply the Bloch-Gieseker construction to $\cO_X(pA)$; then $f\circ\rho$ is projective, the resulting divisor $A'$ is $(f\circ\rho)$-ample, and the trace splitting together with Proposition \ref{P:multi-index V-direct image} reduces the assertion to the covering space. 

As \eqref{E:V twist L id with log comp} and \eqref{E:V twist L id with log comp'}, the cyclic identities and logarithmic comparison identify the
two objects on the covering space with
\[
\left(
\Xi_{\Dol}\cM_1[!C+*B]
\otimes_{\sA_{X'}}^L
(f\circ\rho)^*\sA_S
\right)\otimes\cA^{-1}
\]
and
\[
\left(
\Xi_{\Dol}\cM_{N-1}[!B+*C]
\otimes_{\sA_{X'}}^L
(f\circ\rho)^*\sA_S
\right)\otimes\cA,
\]
respectively. The two vanishings follow from Theorem
\ref{T:relative twistor vanishing} for the projective morphism
$f\circ\rho$. The trace splitting then descends the vanishings
to $X$.
\end{proof}

\bibliographystyle{alpha}
\bibliography{mybib}

\end{document}